\documentclass[10pt,reqno]{amsart}
\usepackage{algorithm,algorithmicx}
\usepackage{algpseudocode}
\usepackage{bbm}
\usepackage[mathlines,pagewise,left]{lineno}
\usepackage{amssymb}
\usepackage{amsmath}
\usepackage{multirow}
\usepackage[numbers]{natbib}
\usepackage{hyperref}
\usepackage{color}
\usepackage{bcol}
\def\bibfont{\small}
\def\bibsep{\smallskipamount}
\def\bibhang{24pt}
\def\newblock{\ }
\def\BIBand{and}
\usepackage{comment}
\usepackage[toc,page]{appendix}
\usepackage{graphicx} % Allows including images
\usepackage{booktabs} % Allows the use of \toprule, \midrule and \bottomrule in tables
\usepackage[colorinlistoftodos,prependcaption,textsize=small]{todonotes}

\usepackage{subcaption}
\usepackage{caption}

\def\SingleSpacedXI{\linespread{1.05}}

\ignore{
\topmargin=-0.2in      %top margin is 1/2 inch (note negative value)
\oddsidemargin=0.25in    %left margin is 1 inch on right-hand pages
\evensidemargin=0.25in   %same for left-hand pages in two sided
\textwidth=6.0in      %leaves 1 inch for the right margin
\textheight=8.7in       %9 inches reserved for the text.
}

\usepackage{amsthm}
\newtheorem{proposition}{Proposition}

\newtheorem{corollary}{Corollary}
\theoremstyle{definition}

\theoremstyle{remark}
\newtheorem{remark}{Remark}

\newcommand{\R}{\ensuremath{\mathbb{R}}}
\newcommand{\Z}{\ensuremath{\mathbb{Z}}}
\newcommand{\N}{\ensuremath{\mathbb{N}}}
\newcommand{\B}{\ensuremath{\mathbb{B}}}
\DeclareMathOperator*{\argmin}{arg\,min}
\newcommand{\CO}{\ensuremath{\text{CO}}}
\newcommand{\MICO}{CDO}
\newcommand{\PO}{\ensuremath{\text{PO}}}
\newcommand{\0}{\ensuremath{\textbf{0}}}
\newcommand{\1}{\ensuremath{\textbf{1}}}
\newcommand{\rev}[1]{{#1}}

\title[Simplex QP-based methods for minimizing a conic quadratic objective]{Simplex QP-based methods for  minimizing \\ a conic quadratic objective over polyhedra}
\author{Alper Atamt\"urk and Andr\'{e}s G\'{o}mez}

\thanks{ \noindent \hskip -5mm
	A. Atamt\"urk: Industrial Engineering \& Operations Research, University of California, Berkeley, CA 94720-1777.
	\texttt{atamturk@berkeley.edu}   \\
	A. G\'{o}mez: Department of Industrial Engineering, Swanson School of Engineering, University of Pittsburgh, PA 15261-3077. \texttt{agomez@pitt.edu}
}

\begin{document}

\begin{abstract}

We consider minimizing a conic quadratic objective over a polyhedron. Such problems arise in parametric 
value-at-risk minimization, portfolio optimization, and  robust optimization with ellipsoidal objective uncertainty; 
and they can be solved by polynomial interior point algorithms for conic quadratic optimization. However,
interior point algorithms are not well-suited for branch-and-bound algorithms for
the discrete counterparts of these problems due to the lack of effective warm starts necessary for the efficient
solution of convex relaxations repeatedly at the nodes of the search tree. 

In order to overcome this shortcoming, we reformulate the problem using the perspective of \rev{the quadratic function}. The perspective 
reformulation lends itself to simple coordinate descent and bisection algorithms utilizing the simplex 
method for quadratic programming, which makes the solution methods amenable to warm starts and suitable 
for branch-and-bound algorithms. We test the simplex-based quadratic programming algorithms to solve convex 
as well as discrete instances and compare them with the state-of-the-art approaches. 
The computational experiments indicate 
that the proposed algorithms scale much better than interior point algorithms and return higher precision solutions.
In our experiments, for large convex instances, they provide up to 22x speed-up. 
For smaller discrete instances, the speed-up is about 13x over a barrier-based branch-and-bound 
algorithm and 6x over the
LP-based branch-and-bound algorithm with extended formulations.\\

\end{abstract}

\maketitle

\begin{center} May 2017; May 2018 \end{center}

\BCOLReport{17.02}%{Mathematical Programming}

\pagebreak
\section{Introduction}

Consider the minimization of a conic quadratic function over a polyhedron, i.e.,
\begin{equation*}
%\label{eq:CP}
(\CO) \  \ \ \min_{x\in \R^n }\left\{c'x+\Omega\sqrt{x'Qx}: x \in X \right\},
\end{equation*}
where $c \in \R^n, \ Q \in \R^{n \times n}$ is a symmetric positive semidefinite matrix, $\Omega>0$, and $X \subseteq \R^n$ is a rational polyhedron.
We denote by \MICO \ the discrete counterpart of \CO \ with integrality restrictions: $X \cap \Z^n$. \CO \ and \MICO \ are frequently used to model utility with uncertain objectives as in parametric value-at-risk minimization \citep{EOO:worst-var}, portfolio optimization \citep{AJ:lifted-polymatroid}, and robust counterparts of linear programs with an ellipsoidal objective uncertainty set \citep{BenTal1998,BenTal1999,book:ro}.

Note that \CO \ includes linear programming (LP) and convex quadratic programming (QP) as special cases. The simplex method \citep{Dantzig1955,Wolfe1959,VanDePanne1964} is still the most widely used algorithm for LP and QP, despite the fact that polynomial interior point algorithms \citep{Karmarkar1984,Nesterov1994,Nemirovskii1996} are competitive with the simplex method in many large-scale instances. Even though  non-polynomial, the simplex method has some distinct advantages over interior point methods. Since the simplex method iterates over bases, it is possible to carry out the computations with  high accuracy and little cost, while interior point methods come with  a trade-off between precision and efficiency. Moreover, an optimal basis returned by the simplex method is useful for sensitivity analysis, while interior point methods do not produce such a basis unless an additional ``crashing" procedure is performed \citep[e.g.][]{Megiddo1991}. Finally, if the parameters of the problem change, re-optimization can often be done very fast with the simplex method starting from a primal or dual feasible basis, whereas warm starts with interior point methods have limitations \citep{YW:warmstart,CPT:warmstart}. 
In particular, fast re-optimization with the dual simplex method is crucial when solving discrete optimization problems with a  branch-and-bound algorithm.

\CO \ is a special case of conic quadratic optimization \citep{Lobo1998,Alizadeh2003}, which can be solved by polynomial-time interior points algorithms \citep{Alizadeh1995,Nesterov1998,BTN:ModernOptBook}.
Although \CO \ can be solved by a general conic quadratic solver, we show in this paper that iterative QP algorithms scale much better. In particular, simplex-based QP algorithms allowing warm starts perform \rev{much} faster than interior point methods for \CO.
%\todo{Do we need to adjust the claim?} %Moreover, they are more suitable as convex relaxation solvers at the nodes of branch-and-bound algorithms for the discrete counterpart \MICO.

For the discrete counterpart \MICO,  a number of different approaches are available for the special case with a diagonal $Q$ matrix: \citet{Ishii1981} give a polynomial time for optimization over spanning trees; \citet{Bertsimas2004} propose an approximation algorithm that solves series of linear integer programs;  \citet{Atamturk2008a} give a cutting plane algorithm utilizing the submodularity of the objective for the binary case; \citet{AG:mixed-polymatroid} give nonlinear cuts for the mixed 0-1 case;
\citet{Atamturk2009} give a parametric $O(n^3)$ algorithm for the binary case with a cardinality constraint.
Maximization of the same objective over the binaries is \NP-hard \cite{AA:utility}.

%\pagebreak

The aforementioned approaches do not extend to the non-diagonal case or to general feasible regions, which are obviously \NP-hard
as quadratic and linear integer optimization are special cases.
The branch-and-bound algorithm is the method of choice for general \MICO. 
However, branch-and-bound algorithms that repeatedly employ a nonlinear programming (NLP) solver at the nodes of the search tree are typically hampered by the lack of effective warm starts.  \citet{Borchers1994} and \citet{Leyffer2001} describe NLP-based branch-and-bound algorithms, and they give methods that branch without solving the NLPs to optimality, reducing the computational burden for the node relaxations. On the other hand, LP-based branch-and-bound approaches employ linear outer approximations of the nonlinear terms. This generally results in weaker relaxations at the nodes, compared to the NLP approaches, but allows one to utilize warm starts with the simplex method. Therefore, one is faced with a trade-off between the strength of the node relaxations and the solve time per node. A key  idea to strengthen the node relaxations, as noted by \citet{Tawarmalani2005}, is to use extended formulations.
\citet{AN:conicmir} describe mixed-integer rounding inequalities in an extended formulation for conic quadratic integer programming.
\citet{Vielma2015} use an extended formulation for conic quadratic optimization that can be refined during branch-and-bound, and show that an LP-based branch-and-bound using the extended formulations typically outperforms the NLP-based branch-and-bound algorithms.
The reader is referred to \citet{jeff-minlp-review} for an excellent survey of the solution methods for mixed-integer nonlinear optimization.
\ignore{\cite{Vielma2008} use the extended formulation for SOCPs proposed by \cite{BenTal2001} to construct a tight initial LP approximation, and \cite{Hijazi2013} use univariate extended formulations for separable MINLPs.}
 
In this paper, we reformulate \CO \ through the perspective of \rev{the quadratic term} and give algorithms that solve a sequence of closely related QPs. Utilizing the simplex method, the solution to each QP is used to warm start the next one in the sequence, resulting in a small number of simplex iterations and fast solution times. Moreover, we show how to incorporate the proposed approach in a branch-and-bound algorithm, efficiently solving the continuous relaxations to optimality at each node and employing warm starts with the dual simplex method. Our computational experiments indicate that the proposed approach outperforms the state-of-the-art algorithms for convex as well as discrete cases. 

The rest of the paper is organized as follows. In Section~\ref{sec:formulation} we give an alternative formulation for \CO \ using the perspective function \rev{of the quadratic function}. In Section~\ref{sec:algorithms} we present coordinate descent and accelerated bisection algorithms that solve a sequence of QPs. In Section~\ref{sec:computational} we provide computational experiments, comparing the proposed methods with state-of-the-art barrier and other algorithms.
% approaches used by commercial algorithms. %In Section~\ref{sec:extensions} we explore how to use the proposed algorithms in other settings. Finally, in Section~\ref{sec:conclusions} we conclude the paper.

\section{Formulation}
\label{sec:formulation}

In this section we present a reformulation of \CO \ using the perspective function of \rev{the quadratic term}. 
Let $X=\left\{x\in\R^{n}:Ax=b, \ x \ge 0 \right\}$ be the feasible region of problem \CO. 
For convex quadratic $q(x) = x'Q x$, consider the function
$h:\R^{n+1}\to \R_+ \cup \{\infty\}$ defined as 
$$h(x,t)=\begin{cases}\frac{x'Qx}{t} & \text{if }t>0,\\ 0 & \text{if }x'Qx = 0, t =0,\\ +\infty & \text{otherwise.}\end{cases}$$
%$$h(x,t)=\begin{cases}\frac{x'Qx}{t} & \text{if }t>0,\\ 0 & \text{if }t=0,\\ +\infty & \text{otherwise.}\end{cases}$$
Observe that 
\begin{align*}
%\tag{CP} 
\nonumber
&\min \left\{c'x+\Omega\sqrt{x'Qx}: x \in X \right\}\\
%\label{eq:redundant}
%\nonumber
%=&\min\left\{c'x+\frac{\Omega}{2}\sqrt{x'Qx}+\frac{\Omega}{2}\sqrt{x'Qx}:Ax=b, x\geq 0,t=\sqrt{x'Qx}\right\}\\
%\label{eq:substitution}
\nonumber
=&\min\left\{c'x+\frac{\Omega}{2}h(x,t)+\frac{\Omega}{2}t : x \in X, \ t=\sqrt{x'Qx}\right\}\\
%\label{PO}
\geq & \ \zeta,
\end{align*}
where
\begin{align*}
%\tag{PO}
%\label{PO}
 (\PO) \ \ \ \zeta = \min \left\{c'x+\frac{\Omega}{2}h(x,t)+\frac{\Omega}{2}t: x \in X, \ t\geq 0\right\}.
\end{align*}
\ignore{ % no need to belabor - clear
The equality in \eqref{eq:redundant} holds since we are only introducing a redundant variable, in \eqref{eq:substitution} we are substituting in the objective, and the inequality in \PO \ holds because we relax the non-convex constraint into a nonnegativity constraint. }

%Note that \PO  is a convex optimization problem as $h$ is perspective of the convex quadratic function $q$ and other terms in th eobjective and constraints are linear. 

We will show that problems \CO \ and \PO \ have, in fact, the same optimal objective value and that there is a one-to-one correspondence between the optimal primal-dual pairs of both problems. 

\begin{proposition}
\label{prop:convexity}
Problem \PO \ is a convex optimization problem.
\end{proposition}
\begin{proof}
It suffices to observe that $h$ is the closure of the \emph{perspective function} $t q(x/t)$
of the convex quadratic function $q(x)$, and is therefore convex \citep[e.g.][p. 160]{book:HUL-conv}. Since all other objective terms and constraints of \PO  \ are linear, \PO \  is a convex optimization problem.
\end{proof}

\begin{proposition}
\label{prop:equivalence}
Problems \CO \ and \PO \ are equivalent.
\end{proposition}
\begin{proof}
If $t >0$, the objective function of problem \PO \ is continuous and differentiable, and since the feasible region is a polyhedron and the problem is convex, its KKT points are equivalent to its optimal solutions. The KKT conditions of \PO \ are
\begin{align}
Ax&=b, \ x\geq 0, \ t\geq 0 \notag\\
\label{eq:KKT1}-c'-\frac{\Omega }{t}x'Q&=\lambda'A-\mu\\
\label{eq:KKT2}\frac{\Omega}{2t^2}x'Qx-\frac{\Omega}{2}&=0\\
\notag\mu&\geq 0\\
\notag\mu' x&=0,
\end{align}
where $\lambda$ and $\mu$ are the dual variables associated with constraints $Ax=b$ and $x\geq 0$, respectively. Note that $t>0$ and \eqref{eq:KKT2} imply that $t=\sqrt{x'Qx}$. Substituting  $t=\sqrt{x'Qx}$ in \eqref{eq:KKT1}, one arrives at the equivalent conditions
\begin{align}
Ax&=b, \ x\geq 0\notag\\
\label{eq:KKT0}-c'-\frac{\Omega}{\sqrt{x'Qx}}x'Q&=\lambda'A-\mu\\
t&=\sqrt{x'Qx}\label{eq:notInteresting}\\
 \mu&\geq 0\notag\\
\mu' x&=0\notag.
\end{align}
Ignoring the redundant variable $t$ and equation \eqref{eq:notInteresting}, we see that these are the KKT conditions of problem \CO. Therefore, any optimal primal-dual pair for \PO \ with $t>0$ is an optimal primal-dual pair for \CO. Similarly, we see that any optimal primal-dual pair of problem \CO \  with $x'Qx>0$ gives an optimal primal-dual pair of problem \PO \ by setting $t=\sqrt{x'Qx}$. In both cases, the objective values match.

On the other hand, if $t=0$, then \PO \ reduces to problem 
\begin{equation*}
\label{eq:CP0}
\min_{x\in \R^{n}}\left\{c'x:Ax=b, x\geq 0,x'Qx=0\right\},
\end{equation*}
which corresponds to \CO \ with $x'Qx = 0$, and hence they are equivalent.
\end{proof}

\ignore{
The objective function of problem \PO is not differentiable when $t=0$ (and the objective function of problem \CP is not differentiable when $x'Qx=0$), and therefore there may be optimal solutions to both problems that are not KKT points. Using the convention that infeasible solutions correspond to an objective value of $\infty$, we see that when $t=0$ problem \PO is equivalent to 
\begin{equation*}
\label{eq:CP0}
\min\left\{c'x: x \in X, \ x'Qx=0\right\}.
\end{equation*}
Therefore we see that the set of feasible solutions of problem \PO with $t=0$ is the same as the set of feasible solution of \CP with $x'Qx=0$, and that such solutions have the same objective value. Therefore, $(x,t)$ with $t=0$ is optimal for \PO if and only if $x'Qx=0$ and $x$ is optimal for \CP. It follows that, in all cases, the set of optimal solutions of \CP and \PO are essentially the same.
}

Proposition~\ref{prop:equivalence} indeed holds for more general problems; it is not necessary to have a polyhedral feasible set \cite{ADJ:mr-interdiction}. 
Since they are equivalent optimization problems, we can use \PO  \ to solve \CO. In particular, we exploit the fact that, for a fixed value of $t$, \PO \ reduces to a QP.

\section{Algorithms}
\label{sec:algorithms}

For simplicity, assume that $\PO$ has an optimal solution; hence, $X$ is nonempty and may be assumed to be bounded.
Consider the one-dimensional optimal value function 
\begin{equation}
\label{eq:oneDimensional}
g(t)=\min_{x\in X}c'x+\frac{\Omega}{2}h(x,t) +\frac{\Omega}{2}t \cdot
\end{equation}
As $X$ is nonempty and bounded, $g$ is real-valued and, by Proposition~\ref{prop:convexity}, it is convex. 
Throughout, $x(t)$ denotes an optimal solution to  \eqref{eq:oneDimensional}.
% and that function $g$ has a unique minimizer $t^*>0$.  
%The results can easily be extended to the case when $g$ has an interval of minimizers, and we discuss the unbounded case in Section~\ref{sec:unbounded}. 
%Let $x^*=\argmin_{x\in X}c'x+\Omega\sqrt{x'Qx}$ and for $t>0$ let $x(t)=\argmin_{x\in X} c'x+\frac{\Omega}{2t}x'Qx$. Observe that $x^*=x(t^*)$ by Proposition~\ref{prop:equivalence}. \todo{optimal solution x*, x(t*) may no be unique. Is this needed?}

In this section we describe two algorithms for \PO \  that utilize a QP oracle. The first one is a coordinate descent approach,
% and is studied in Section~\ref{sec:coordinate}. 
whereas, the second one is an accelerated bisection search algorithm 
on the function $g$. 
% and is studied in Section~\ref{sec:bisection}. 
Finally, %in Section~\ref{sec:warmStarts}, 
we discuss how to exploit the warm starts with the simplex method to solve convex as well as discrete cases.

\subsection{Coordinate descent algorithm}
\label{sec:coordinate}
%Consider the approach described in Algorithm~\ref{alg:coordinateDescent}. 
Algorithm~\ref{alg:coordinateDescent} successively optimizes over $x$ for a fixed value of $t$, and then optimizes over $t$ for a fixed value of $x$. Observe that the optimization problem in line~\ref{line:QP} over $x$ is a QP, and the optimization in line~\ref{line:closedForm} over $t$ has a closed form solution: by simply setting the derivative to zero, we find that $t_{i+1}=\sqrt{{x_{i+1}}'Qx_{i+1}}$.

\begin{algorithm}[h]
\caption{Coordinate descent.}
\label{alg:coordinateDescent}
\begin{algorithmic}[1]
\renewcommand{\algorithmicrequire}{\textbf{Input:}}
\renewcommand{\algorithmicensure}{\textbf{Output:}}
\Require $X \text{ polyhedron; }Q\text{ psd matrix; }c\text{ cost vector; } \Omega>0$
\Ensure Optimal solution $x^*$

\State \textbf{Initialize }$t_0 > 0$ \label{line:initt0} \Comment{e.g. $t_0=1$}
\State $i\leftarrow 0$ \Comment{iteration counter}
\Repeat
\State $x_{i+1}\leftarrow \argmin\limits_{x\in X}\left\{c'x+\frac{\Omega}{2t_i}x'Qx+\frac{\Omega}{2}t_{i}\right\}$\Comment{solve QP}\label{line:QP}
\State $t_{i+1}\leftarrow \argmin\limits_{t\geq 0}\left\{c'x_{i+1}+\frac{\Omega}{2t}{x_{i+1}}'Qx_{i+1}+\frac{\Omega}{2}t\right\}$\Comment{$t_{i+1}=\sqrt{{x_{i+1}}'Qx_{i+1}}$}\label{line:closedForm}
\State $i\leftarrow i+1$
\Until stopping condition is met \label{line:stoppingCriterion}
\State \Return $x_i$
\end{algorithmic}
\end{algorithm}

First observe that the sequence of objective values $\left\{c'x_i+\frac{\Omega}{2t_i}x_i'Qx_i+\frac{\Omega}{2}t_{i}\right\}_{i\in \N}$ is non-increasing. Moreover, the dual feasibility KKT conditions for the QPs in line \ref{line:QP} are of the form
\begin{equation}
\label{eq:QPKKT}
-c'-\frac{\Omega}{t_i}{x_{i+1}}'Q=\lambda'A-\mu.
\end{equation}
Let $\|\cdot\|$ be a norm and suppose that the QP oracle finds feasible primal-dual pairs with $\epsilon>0$ tolerance with respect to $\|\cdot\|$. In particular $x_{i+1}$ in line \ref{line:QP} violates \eqref{eq:QPKKT} by at most $\epsilon$, i.e.,
\begin{equation*}
\left\|-c'-\frac{\Omega}{t_i}{x_{i+1}}'Q-\lambda'A+\mu\right\|\leq \epsilon.
\end{equation*}
Proposition \ref{prop:convergence} below states that, at each iteration of Algorithm~\ref{alg:coordinateDescent}, we can bound the violation of the dual feasibility condition \eqref{eq:KKT0} corresponding to the original problem \CO. The bound depends only on the precision of the QP oracle $\epsilon$, the relative change of $t$ in the last iteration $\frac{\Delta_i}{t_i}$, where $\Delta_i=t_{i+1}-t_i$, and the gradient of the function $f(x)= \Omega \sqrt{x'Qx}$ evaluated at the new point $x_{i+1}$.

\begin{proposition}[\textit{Dual feasibility bound}]
\label{prop:convergence}
A pair $(x_{i+1},t_{i+1})$ in Algorithm~\ref{alg:coordinateDescent} satisfies 
$$\left\|-c'-\Omega\frac{x_{i+1}'Q}{\sqrt{{x_{i+1}}'Qx_{i+1}}}-\lambda'A+\mu\right\| \leq \epsilon+\frac{\left|\Delta_i\right|}{t_i}\cdot 
 \left\| \nabla f(x_{i+1})\right\|$$% \xrightarrow{\Delta_i\to 0}\epsilon.$$
\end{proposition}

\begin{proof}
\begin{align*}
&\left\|-c'-\Omega\frac{ {x_{i+1}}'Q}{\sqrt{{x_{i+1}}'Qx_{i+1}}}-\lambda'A+\mu\right\|\\
%=&\left\|-c'-\Omega\frac{{x_{i+1}}'Q}{t_{i+1}}-\lambda'A+\mu\right\|\\
=&\left\|-c'-\Omega\frac{{x_{i+1}}'Q}{t_i+\Delta_i}-\lambda'A+\mu\right\|\\
=&\left\|-c'-\Omega\frac{{x_{i+1}}'Q}{t_i}-\Omega {x_{i+1}}'Q\left(\frac{1}{t_i+\Delta_i}-\frac{1}{t_i}\right)-\lambda'A+\mu\right\|\\
=&\left\|-c'-\Omega\frac{{x_{i+1}}'Q}{t_i}-\lambda'A+\mu+\Omega \left(\frac{\Delta_i}{t_i\cdot t_{i+1}}\right) {x_{i+1}}'Q \right\|  \\
\leq& \epsilon +\left\| \Omega \frac{\Delta_i}{t_i} \cdot \frac{{x_{i+1}}'Q}{t_{i+1}}\right\|=\epsilon+ \Omega \frac{\left|\Delta_i\right|}{t_i}\cdot  \left\|  \frac{{x_{i+1}}'Q}{\sqrt{{x_{i+1}}'Qx_{i+1}}}\right\|.
\end{align*}
\end{proof}

%Let $t^*$ be any optimal solution of \eqref{eq:oneDimensional}. 
Let $t^*$ be a minimizer of $g$ on $\R_+$. 
We now show that the sequence of values of $t$ produced by Algorithm~\ref{alg:coordinateDescent}, 
$\left\{t_i\right\}_{i\in \N}$, is monotone and bounded by $t^*$.
\begin{proposition}[\textit{Monotonicity}] 
\label{prop:monotonicity}
If $t_i\leq t^*$, then $t_{i+1}=\sqrt{{x_{i+1}}'Qx_{i+1}}$ satisfies $t_i\leq t_{i+1}\leq t^*$. Similarly, if $t_i\geq t^*$, then $t_i\geq t_{i+1}\geq t^*$.
\end{proposition}
\begin{proof}

If $t_i\leq t^*$, then $\frac{\Omega}{2t_i}\geq \frac{\Omega}{2t^*}$. It follows that $x(t_{i+1})$ is a minimizer of an optimization problem with a larger coefficient for the quadratic term than $x(t^*)$, and therefore ${{x_{i+1}}'Qx_{i+1}}=t_{i+1}^2\leq {t^*}^2= {x^*}'Qx^*$, and $t_{i+1}\leq t^*$. Moreover, the inequality $t_i\leq t_{i+1}$ follows from the convexity of the one-dimensional function $g$ and
 the fact that function $g$ is minimized at $t^*$, and that $g(t_{i+1})\leq g(t_i)$.
The case $t_i\geq t^*$ is similar.
 \end{proof}

Since the sequence $\left\{t_i\right\}_{i\in \N}$ is bounded and monotone, it converges to a supremum or infimum. Thus  $\left\{t_i\right\}_{i\in \N}$ is a Cauchy sequence, and
$\lim\limits_{i \to \infty} \Delta_i = 0$. Corollaries \ref{cor:KKTConvergence} and \ref{cor:0Convergence} below state that Algorithm~\ref{alg:coordinateDescent} converges to an optimal solution. The cases where there exists a KKT point for \PO \ (i.e., there exists an optimal solution with $t^*>0$) and where there are no KKT points are handled separately.

\begin{corollary}[Convergence to a KKT point]
\label{cor:KKTConvergence}
If \PO \ has a KKT point, then Algorithm~\ref{alg:coordinateDescent} converges to a KKT point.
\end{corollary}
\begin{proof}
By convexity, the set of optimal solutions to \eqref{eq:oneDimensional} is an interval, $[t_\ell,t_u]$. Since by assumption there exists a KKT point, we have that $t_u>0$. The proof is by cases, depending on the value of $t_0$ in line~\ref{line:initt0} of Algorithm~\ref{alg:coordinateDescent}.
\begin{description}
\item [Case $t_\ell\leq t_0\leq t_u$] Since $t_0$ is optimal, we have by Proposition~\ref{prop:monotonicity} that $t_1=t_0$. Since $\Delta_0=0$ and $t_0=\sqrt{x_{i+1}'Qx_{i+1}}>0$, we have that $\left\| \nabla f(x_{i+1})\right\|<\infty$ in Proposition~\ref{prop:convergence}, and $\frac{\left|\Delta_i\right|}{t_i}\cdot 
 \left\| \nabla f(x_{i+1})\right\|=0$.
\item [Case $t_0< t_\ell$]We have by Proposition~\ref{prop:monotonicity} than for all $i\in \N$, $t_i=\sqrt{x_i'Qx_i}\geq t_0>0$. Therefore, there exists a number $M$ such that $\frac{1}{t_i}\left\| \nabla f(x_{i+1})\right\|<M$ for all $i\in \N$, and we find that $\frac{\left|\Delta_i\right|}{t_i}\cdot 
 \left\| \nabla f(x_{i+1})\right\|\xrightarrow{\Delta_i\to 0} 0$.
\item [Case $t_0> t_u$]We have by Proposition~\ref{prop:monotonicity} than for all $i\in \N$, $t_i=\sqrt{x_i'Qx_i}\geq t_u>0$. Therefore, there exists a number $M$ such that $\frac{1}{t_i}\left\| \nabla f(x_{i+1})\right\|<M$ for all $i\in \N$, and we find that $\frac{\left|\Delta_i\right|}{t_i}\cdot 
 \left\| \nabla f(x_{i+1})\right\|\xrightarrow{\Delta_i\to 0} 0$.
\end{description}
Therefore, in all cases, Algorithm~\ref{alg:coordinateDescent} convergences to a KKT point by Proposition~\ref{prop:convergence}.
\end{proof}

\begin{corollary}[Convergence to $0$]
\label{cor:0Convergence}
If $t^*=0$ is the unique optimal solution to $\min \{g(t): t \in \R_+\}$, then for any $\xi>0$ Algorithm~\ref{alg:coordinateDescent} finds a solution $(\bar{x},\bar{t})$, where $\bar{t}<\xi$ and $\bar{x}\in \argmin\left\{c'x:\sqrt{x'Qx}=\bar{t}, x\in X\right\}$.
\end{corollary}
\begin{proof}
The sequence $\left\{t_i\right\}_{i\in \N}$ converges to $0$ (otherwise, by Corollary~\ref{cor:KKTConvergence}, it would converge to a KKT point). Thus, $\lim_{i\to\infty}\sqrt{x_i'Qx_i}=0$ and all points obtained in line~\ref{line:QP} of Algorithm~\ref{alg:coordinateDescent} satisfy $x_{i+1}\in \argmin\left\{c'x:\sqrt{x'Qx}=t_{i+1}, x\in X\right\}$.
\end{proof}

%We now show that, as $\{t_i\}_{i \in N} \rightarrow t^*$, so does $\{x_i\}_{i \in N} \rightarrow x^*$.

%as it converges, the corresponding sequence $\left\{x_i\right\}_{i\in \N}$ approaches an optimal solution of problem \CP.

\ignore{
\begin{remark}
From Proposition~\ref{prop:convergence} we see that optimal primal-dual pairs of  \CP correspond to the optimal primal-dual pairs of the QP \eqref{eq:oneDimensional} at $t^*$.
\end{remark}
}

We now discuss how to initialize and terminate Algorithm~\ref{alg:coordinateDescent}, corresponding to lines \ref{line:initt0} and \ref{line:stoppingCriterion}, respectively.
\subsubsection*{Initialization.}
% The initialization in line~\ref{line:initt0} can be done arbitrarily. 
 The algorithm may be initialized by an arbitrary $t_0 > 0$.
 Nevertheless, when a good initial guess on the value of $t^*$ is available, $t_0$ should be set to that value. 
 Moreover, observe that setting $t_0=\infty$ results in a fast computation of $x_1$ by solving an LP.

\subsubsection*{Stopping condition.}
Proposition~\ref{prop:convergence} suggests a good stopping condition for Algorithm~\ref{alg:coordinateDescent}. Given a desired dual feasibility tolerance of $\delta>\epsilon$, we can stop when $\epsilon + \frac{\left|\Delta_i\right|}{t_i}\cdot \left\| \nabla f(x_{i+1}) \right\|<\delta$. Alternatively, if 
$\exists k \text{ s.t. }  \max_{x \in X} \left\| \nabla f(x) \right\| \le k < \infty$, then the simpler $\left|\frac{\Delta_i}{t_i}\right|\leq \frac{\delta-\epsilon}{k}$ is another stopping condition. For instance,
a crude upper bound on $  \rev{\|}\nabla f(x) \rev{\|}= \Omega\left\| \frac{{x}'Q}{\sqrt{{x}'Qx}}\right\|$ can be found by maximizing/minimizing the numerator $x'Q$ over $X$ and minimizing $x'Qx$ over $X$. The latter minimization is guaranteed to have a nonzero optimal value if $0 \not \in X$ and $Q$ is positive definite.
%\end{remark}

\rev{
\begin{remark}
Observe that the stopping condition above may fail if $t^*=0$ is the unique optimal solution of $\min_{t\geq 0}g(t)$ (Corollary~\ref{cor:0Convergence}). This case happens only if $Q$ is positive semi-definite (but not positive definite), or if $x^*=0$ is the unique optimal solution of \CO. Such situations rarely arise in practice and can often be ruled out \emph{a priori}. Nonetheless, stopping Algorithm~\ref{alg:coordinateDescent} also when $t_i\leq \xi$ for some small $\xi>0$ ensures that the algorithm terminates (as specified in Corollary~\ref{cor:0Convergence}) even when $t^*=0$. 
\end{remark}
}

\ignore{
\begin{remark}
We provide some intuition for Proposition~\ref{prop:convergence}. Recall that $t_i=\sqrt{x_i'Qx}$, and so we can write (with an abuse of notation) that the gradient of $t$ at $x$ is $\frac{\partial t}{\partial x}(x_i)=\frac{{x_i}'Q}{\sqrt{{x_i}'Qx_i}}$. A natural estimator of the future change of $t$ is the rate of change of $t$ at the current point, given by $\Omega\frac{\partial t}{\partial x}(x_{i+1})$, times the relative change in the previous iteration, $\frac{\Delta_i}{t_i}$. According to Proposition~\ref{prop:convergence}, the natural estimator gives a bound on the violation of KKT condition \eqref{eq:KKT0} at the current point.
\end{remark}
}

\subsection{Bisection algorithm}
\label{sec:bisection}
%Consider the approach described in Algorithm~\ref{alg:bisection}. 

Algorithm~\ref{alg:bisection} is an accelerated bisection approach to solve \PO. The algorithm maintains lower and upper bounds, $t_{\min}$ and $t_{\max}$, on $t^*$ and, at each iteration, reduces the interval $[t_{\min}, t_{\max}]$ by at least half. The algorithm differs from the traditional bisection search algorithm in lines \ref{line:iBisection10}--\ref{line:iBisection3}, where it uses an acceleration step to reduce the interval by a larger amount: 
by Proposition~\ref{prop:monotonicity},
if $t_0\leq t_1$ (line \ref{line:iBisection10}), then $t_0\leq t_1\leq t^*$, and therefore $t_1$ is a higher lower bound on $t^*$ (line \ref{line:iBisection11}); similarly, if $t_0\geq t_1$, then $t_1$ is an lower upper bound on $t^*$ (lines \ref{line:iBisection20} and \ref{line:iBisection21}). Intuitively, the algorithm takes a ``coordinate descent" step as in Algorithm~\ref{alg:coordinateDescent} after each bisection step. Preliminary computations show that the acceleration step reduces
the number of steps as well as the overall solution time for the bisection algorithm by about 50\%.

\begin{algorithm}[h]
\caption{Accelerated bisection.}
\label{alg:bisection}
\begin{algorithmic}[1]
\renewcommand{\algorithmicrequire}{\textbf{Input:}}
\renewcommand{\algorithmicensure}{\textbf{Output:}}
\Require $X \text{ polyhedron; }Q\text{ psd matrix; }c\text{ cost vector; } \Omega>0$
\Ensure Optimal solution $x^*$
\State \textbf{Initialize }$t_{\min}$ and $t_{\max}$ \Comment{ensure $t_{\min}\leq t^* \leq t_{\max}$}\label{line:initTs}
\State $\hat{z}\leftarrow \infty$ \Comment{best objective value found}
%\State $\hat{x}\gets \emptyset$ \Comment{best solution found}
\Repeat
\State $t_0\leftarrow \frac{t_{\min}+t_{\max}}{2}$
\State $x_0\leftarrow \argmin\limits_{x\in X}\left\{c'x+\frac{\Omega}{2t_0}x'Qx+\frac{\Omega}{2}t_{0}\right\}$\Comment{solve QP}\label{line:updateX}
\State $t_{1}\leftarrow \sqrt{{x_{0}}'Qx_{0}}$
\If{$t_0 \leq t_1$}\label{line:iBisection10} \Comment{accelerate bisection}
\State $t_{\min}\leftarrow t_1$\label{line:iBisection11}
\Else\label{line:iBisection20}
\State $t_{\max}\leftarrow t_1$\label{line:iBisection21}
\EndIf \label{line:iBisection3}
\If{$c'x_0+\Omega\sqrt{{x_0}'Qx_0}\leq \hat{z}$} \Comment{update the incumbent solution}
\State $\hat{z}\leftarrow c'x_0+\Omega\sqrt{{x_0}'Qx_0}$
\State $\hat{x}\leftarrow x_0$
\EndIf
\Until stopping condition is met \label{line:stoppingCriterion2}
\State \Return $\hat{x}$
\end{algorithmic}
\end{algorithm}

%\begin{remark} 
\subsubsection*{Initialization.}
In line~\ref{line:initTs}, $t_{\min}$ can be initialized to zero and $t_{\max}$ to ${x_{LP}}'Qx_{LP}$, where $x_{LP}$ is an optimal solution to the LP relaxation  
$\min_{x\in X}c'x$.

%\begin{remark} 
\subsubsection*{Stopping condition.}
There are different possibilities for the stopping criterion in line \ref{line:stoppingCriterion2}. Note that if we have numbers $t_m$ and $t_M$ such that $t_m \leq t^* \leq t_M$, then $c'x (t_M)+\Omega\sqrt{{x(t_m)}'Qx(t_m)}$ is a lower bound on the optimal objective value $c'x^*+\Omega\sqrt{{x^*}'Qx^*}$. Therefore, in line~\ref{line:updateX}, a lower bound $z_l$ on the objective function can be computed, and the algorithm can be stopped when the gap between $\hat{z}$ and $z_l$ is smaller than a given threshold. Alternatively, stopping when $\frac{\left|t_1-t_0\right|}{t_0}\cdot \Omega\left\| \frac{{x_{0}}'Q}{\sqrt{{x_{0}}'Qx_{0}}}\right\|<\delta-\epsilon$ provides a guarantee on the dual infeasibility as in Proposition~\ref{prop:convergence}. 
%\end{remark}

\subsection{Warm starts}
\label{sec:warmStarts} 
Although any QP solver can be used to run the coordinate descent and bisection algorithms described in Sections \ref{sec:coordinate} and \ref{sec:bisection}, simplex methods for QP are particularly effective 
as they allow warm starts for small changes in the model parameters in iterative applications. This is the main motivation for the QP based algorithms presented above.

\subsubsection{Warm starts with primal simplex for convex optimization}
\label{sec:warmStartPrimal}
All QPs solved in Algorithms~\ref{alg:coordinateDescent}--\ref{alg:bisection} have the same feasible region and only the objective function changes in each iteration. Therefore, an optimal basis for a QP is primal feasible for the next QP solved in the sequence, and can be used to warm start a primal simplex QP solver. 

\subsubsection{Warm starts with dual simplex for discrete optimization}
When solving discrete counterparts of \CO \ with a branch-and-bound algorithm
one is particularly interested in utilizing warm starts in solving convex relaxations at the nodes of the search tree. In a branch-and-bound algorithm, children nodes typically have a single additional bound constraint compared to the parent node.

For this purpose, it is also possible to warm start Algorithm~\ref{alg:coordinateDescent} from a dual feasible basis. 
Let $(x^*,t^*)$ be an optimal solution to \PO \  and $B^*$ be an optimal basis. Consider a new problem
\begin{equation}
\label{eq:dualFeasible}
\min \left\{c'x+\frac{\Omega}{2t}x'Qx+\frac{\Omega}{2}t: x \in \bar X, \ t \ge 0\right\},
\end{equation} 
where the feasible set $\bar{X}$ is obtained from $X$ by adding new constraints. 
%and let $(\bar{x^*},\bar{t^*})$ be an optimal solution for problem \eqref{eq:dualFeasible}. 
Note that $B^*$ is a dual feasible basis for \eqref{eq:dualFeasible} when $t = t^*$. Therefore, 
Algorithm~\ref{alg:coordinateDescent} to solve problem \eqref{eq:dualFeasible} can be warm started 
by initializing $t_0=t^*$ and using $B^*$ as the initial basis to compute $x_1$ with a dual simplex algorithm. 
The subsequent QPs can be solved using the primal simplex algorithm as noted in Section~\ref{sec:warmStartPrimal}.

%(line~\ref{line:QP}).  
%Intuitively, if the change in the constraint matrix is small, then the quantity $|t^*-\bar{t^*}|$ is small and the starting point of the algorithm is thus very accurate. 
\ignore{
In typical branch-and-bound algorithms for MILPs and MIQPs, the optimal basis found at each node is then used to warm start the continuous solver in the children nodes. A child node typically has a single additional bound constraint. To extend the branch-and-bound algorithms to MICPs, it is sufficient to define the basis as the pair $(B^*,t^*)$ described in the previous paragraph, and use Algorithm~\ref{alg:coordinateDescent} as the continuous solver. 
%Also note that the space required to store a CP basis is only marginally greater than the space required to store a QP basis (corresponding to storing the value of $t^*$).
}

\rev{
\subsection{Special cases} The simplex method is a general algorithm that can be used with any polyhedron $X$ and, as mentioned in Section~\ref{sec:warmStarts}, is well suited for solving the sequence of QPs. Nevertheless, for particular feasible regions, other specialized algorithms may be preferable. For instance, in the trivial unbounded case ($X=\R^n$), the QPs can be solved in closed form. Another, more interesting, case is the problem
\begin{equation}
\label{eq:sqrtLasso}
\min_{x\in \R^n}\sqrt{(x-y)'Q(x-y)}+\beta\|x\|_1,
\end{equation}
where $y\in \R^n$ is fixed, $\|\cdot\|_1$ denotes the $\ell_1$-norm and $\beta>0$ is a regularization parameter. Note that \eqref{eq:sqrtLasso} is a special case of \CO \ with the usual linearization of the $\ell_1$-norm. Problem \eqref{eq:sqrtLasso} arises in compressed sensing \citep{aybat2011first} and sparse linear regression \citep{belloni2011square}, and is solved fast using first-order methods \citep{nesterov2005smooth}. Note if Algorithm~\ref{alg:coordinateDescent} or \ref{alg:bisection} is used instead, then every QP subproblem $\min_{x\in \R^n}\frac{(x-y)'Q(x-y)}{2t}+\beta\|x\|_1$ corresponds to the well-studied Lasso problem \citep{tibshirani1996regression}, for which efficient special algorithms exist. In particular, problem \eqref{eq:sqrtLasso} can be solved with a \emph{single} call to an algorithm that computes the regularization path (i.e., solves the problem for all $\beta$), such as Least Angle Regression \citep{efron2004least}.
}

\ignore{
\subsection{Unbounded case}
\label{sec:unbounded}
In many cases it is possible to determine that problem \CP \ is bounded \textit{a priori} (e.g., $X$ is a polytope, or $c\geq 0$). We now discuss how to detect whether problem \CP is bounded or not when there is not a simple guarantee.

First, note that if the LP relaxation \eqref{eq:LPrelaxation} is bounded then problem \CP is bounded. Moreover, as Proposition~\ref{prop:unbounded} states, if any of the QPs is unbounded then problem \CP is unbounded. 

\begin{proposition}
\label{prop:unbounded}
If $g(t)=-\infty$ for any fixed $t\geq 0$, then problem \CP is unbounded.
\end{proposition}
\begin{proof}
If $g(t)=-\infty$, then there exists a sequence of feasible points $\left\{x_i\right\}_{i\in \N}$ such that 
\begin{align*}
&\lim_{i\to \infty }c'x_i+\frac{\Omega}{2t}{x_i}'Qx_i=-\infty\\
\implies&\lim_{i\to \infty }c'x_i+\max\left\{\frac{\Omega}{2t}{x_i}'Qx_i,2t\right\}=-\infty.
\end{align*}
Since $c'x+\Omega\sqrt{{x_i}'Qx_i}\leq c'x+\max\left\{\frac{\Omega}{2t}{x_i}'Qx_i,2t\right\}$, we have that the sequence $\left\{x_i\right\}_{i\in \N}$ is also a unbounded sequence for problem \CP.
\end{proof}

Unfortunately, as Example~\ref{ex:unbounded} shows, it is possible that $g(t)>-\infty$ for all $t$ and that problem \CP is unbounded. In this case we have that $\lim\limits_{t\to \infty}g(t)=-\infty$.
%, and that $g(t)$ is a upper bound on the optimal objective function for any value of $t$. 
\begin{example}
\label{ex:unbounded}
Consider the one-dimensional unconstrained problem $$\min_{x\in \R} x+\Omega\left|x\right|,$$
which is unbounded for $\Omega<1$. In this case we have $$g(t)=\min_{x\in \R}\left(x+\frac{\Omega}{2t} x^2+\frac{\Omega}{2}t\right)=t\left(\frac{\Omega^2-1}{2\Omega}\right),$$ which is bounded for all $\Omega> 0$. Nevertheless, we see that when $\Omega<1$ we have that $\lim\limits_{t\to \infty}g(t)=-\infty$.
\end{example}

We now summarize a process for instances that may be unbounded. We first check for easy certificates of boundedness or unboundedness. In case we are unable to verify whether the problem is bounded or not, we run Algorithm~\ref{alg:coordinateDescent} until a feasible solution with a sufficiently low objective value is found.
\begin{description}
\item[Step 1] Solve the LP relaxation \eqref{eq:LPrelaxation}. If it is bounded, then problem \CO is bounded and can be solved using Algorithms\footnote{Note that Algorithm~\ref{alg:bisection} requires solving the LP in any case. Moreover, Algorithm~\ref{alg:coordinateDescent} can be warm started from the LP optimal solution.}~\ref{alg:coordinateDescent} or \ref{alg:bisection}. Otherwise go to Step 2.
\item[Step 2] Initialize $t$, and compute $g(t)$. If $g(t)=-\infty$, then problem \CP is unbounded. Otherwise go to Step 3.
\item[Step 3] Choose a lower bound $m$. Use Algorithm~\ref{alg:coordinateDescent} until convergence (in which case the solution found is optimal) or until a feasible solution is found such that the objective value is less than $m$. 
\end{description}
}

\section{Computational experiments}
\label{sec:computational}

In this section we report on computational experiments with solving convex \CO \ and its discrete counterpart \MICO \ with the algorithms described in Section~\ref{sec:algorithms}. The algorithms are implemented with CPLEX Java API. We use the simplex and barrier solvers of CPLEX version 12.6.2\rev{, as well as the barrier solver of MOSEK version 8.1.0} for the computational experiments. All experiments are conducted on a workstation with a 2.93GHz Intel\textregistered Core\textsuperscript{TM} i7 CPU and 8 GB main memory using a single thread. %We use CPLEXs default settings unless specified otherwise.

\subsection{Test problems} We test the algorithms on two types of data sets. For the first set the feasible region is described by a cardinality constraint and bounds, i.e., $X=\left\{x\in\R^{n}:\sum_{i=1}^n x_i= b,\; 
\0 \leq x \leq \1 \right\}$ with $b = n/5$\rev{; problems with a cardinality constraint are common in finance \citep{bienstock1996computational} and statistics \cite{bertsimas2016best}}. For the second data set the feasible region consists of the 
path polytope of an acyclic grid network\rev{; conic quadratic optimization over paths has been studied in \cite{Bertsimas2004,nikolova2006stochastic}, and similar substructures arise in more complex problems such as vehicle routing \cite{dinh2016exact}.} For discrete optimization problems we additionally enforce the binary restrictions $x\in \B^n$. 

%\rev{We limit our computational experiments to integral polytopes because our branch-and-bound algorithm does not use cutting planes and would not be effective for non-integral polytopes. Note that the lack of cutting planes is a limitation only of our branch-and-bound algorithm, and that Algorithm~\ref{alg:coordinateDescent} could be used in branch and cut approaches.}

\ignore{
\subsubsection{Feasible regions} We consider two classes of feasible regions:
\begin{description}
	\item[Cardinality instances] The feasible region consists of a single cardinality constraint and bound constraints, i.e. $$X=\left\{x\in\R^{n}:\sum_{i=1}^n x_i= b,\; 0\leq x_i\leq 1 \;\forall i=1,\ldots,n\right\}.$$ In the computational experiments, we set $b=n/5$.
	\item[Path instances] The feasible region consists of the path polytope in acyclic grid networks.
\end{description}

}

For both data sets the objective function $q(x) = c'x + \Omega \sqrt{x'Qx}$ is generated as follows:
 Given a rank parameter $r$ and density parameter $\alpha$, $Q$ is the sum of a low rank factor matrix and a full rank diagonal matrix; that is, $Q=F\Sigma F'+D$, where
\begin{itemize}
	\item $D$ is an $n\times n$ diagonal matrix with entries drawn from Uniform$(0,1)$.
	\item $\Sigma=HH'$ where $H$ is an $r\times r$ matrix with entries drawn from Uniform$(-1,1)$.
	\item $F$ is an $n\times r$ matrix in which each entry is $0$ with probability $1-\alpha$ and 
	drawn from Uniform$(-1,1)$ with probability $\alpha$. 
\end{itemize} 
\rev{Note that the construction of matrix $Q$ is consistent with factor models often used in finance. In particular, $F$ is the factor exposure matrix, $\Sigma$ is the factor covariance matrix and $D$ is the matrix of the residual variances.}
Each linear coefficient $c_i$ is drawn from Uniform$(-2\sqrt{Q_{ii}},0)$. 

\ignore{Therefore if the objective function is interpreted as the value-at-risk of normally distributed random variables, then we have that on average the expected return of each variable is proportional to its standard deviation (and risky variables have thus better expected returns).}

\subsection{Experiments with convex problems}
\label{sec:resultsContinuous}

In this section we present the computational results for convex instances. We compare the following algorithms:
\begin{description}
\item [ALG1] Algorithm~\ref{alg:coordinateDescent}.
% (i.e., the first problem in the sequence is an LP).
\item [ALG2] Algorithm~\ref{alg:bisection}.
\item [BAR] CPLEX barrier algorithm (the default solver in CPLEX for convex conic quadratic problems). 
\rev{\item[MOS] MOSEK barrier algorithm.}
\end{description}
For algorithms ALG1 and ALG2 we use CPLEX primal simplex algorithm as the QP solver.\ignore{, and the stopping condition $\frac{\left|\Delta_i\right|}{t}\leq 10^{-5}$
unless specified otherwise.}

%\subsection{Results for the continuous instances}
%\label{sec:resultsContinuous}
\ignore{
 Specifically, we present three sets of computational results. First in Section~\ref{sec:resultsContinuousQ} we study the effects of changing the $Q$ matrix, and we are primarily concerned with comparing between the simplex-based algorithms. Then in Section~\ref{sec:resultsContinuousDimension} we study the effects of changing the dimension (for a fixed structure of the $Q$ matrix), and we are primarily concerned with comparing the performance of the barrier algorithm and the simplex-based algorithms. Finally in Section~\ref{sec:resultsContinuousTolerance} we study the effects of changing the tolerance (for a fixed dimension and structure of the $Q$ matrix), and compare the barrier algorithm with a simplex-based algorithm.

}

\subsubsection*{Optimality tolerance}
%\label{sec:resultsContinuousTolerance}

As the speed of the interior point methods crucially depends on the chosen optimality tolerance, it is prudent to first compare the speed vs the quality of the solutions for the algorithms tested. Here we study the impact of the optimality tolerance in the solution time and the quality of the solutions for CPLEX barrier algorithm BAR and simplex QP-based algorithm  ALG1. The optimality tolerance of the barrier algorithm is controlled by the QCP convergence tolerance parameter (``BarQCPEpComp"), and in Algorithm~\ref{alg:coordinateDescent}, by the stopping condition $\frac{\left|\Delta_i\right|}{t}\leq \delta$. 

In both cases, a smaller optimality tolerance corresponds to a higher quality solution. We evaluate the quality of a solution as $\texttt{optgap}=\left|(z_{\min} -z)/z_{\min}\right|,$
where $z$ is the objective value of the solution found by an algorithm with a given tolerance parameter and $z_{\min}$ is the objective value of the solution found by the barrier algorithm with tolerance $10^{-12}$ (minimum tolerance value allowed by CPLEX). 
Table~\ref{tab:tolerance} presents the results for different tolerance values 
for a $30\times 30$ convex grid instance with $r=200$, $\alpha=0.1$, and $\Omega=1$. 
% (our results indicate the results are similar for instances with different parameters). 
The table shows, for varying tolerance values and for each algorithm, the quality of the solution, the solution time in seconds, the number of iterations, and QPs solved (for ALG1). We highlight in bold the default tolerance used for the rest of the experiments
presented in the paper. The tolerance value $10^{-7}$ for the barrier algorithm corresponds to the default parameter in CPLEX.

{
%\SingleSpacedXI
\renewcommand\arraystretch{1.00}
\begin{table}[h!]
\caption{The effect of optimality tolerance.}
\begin{center}
\label{tab:tolerance}
\SingleSpacedXI
\scalebox{0.9}{
\begin{tabular}{ c | c c c | c c c c} 
\hline \hline
\multirow{2}{*}{\textbf{\texttt{Tolerance}}} &\multicolumn{3}{c|}{\texttt{BAR}}&\multicolumn{4}{c}{\texttt{ALG1}} \\
&$\texttt{optgap}$&\texttt{time}&\texttt{\#iter}&\texttt{optgap}&\texttt{time}&\texttt{\#iter}&\texttt{\#QP} \\
\hline
$10^{-1}$ & $8.65 \times 10^{-2}$&29.9 & 10 & $5.48\times 10^{-5}$ & 3.2 &835 & 4\\
$10^{-2}$ & $8.77 \times 10^{-3}$&41.5 & 15 & $3.24\times 10^{-7}$ & 4.2 &844 & 6\\
$10^{-3}$ & $6.98 \times 10^{-4}$&54.6 & 23 & $2.60\times 10^{-9}$ & 4.3 &844 & 8\\
$10^{-4}$ & $5.52 \times 10^{-5}$&62.9 & 27 & $2.12\times 10^{-11}$ & 4.7 &844 & 10\\
$10^{-5}$ & $3.72 \times 10^{-6}$&66.8 & 29 & $\boldsymbol{6.80\times 10^{-13}}$ & \textbf{5.2} &\textbf{844} & \textbf{12}\\
$10^{-6}$ & $7.12 \times 10^{-7}$&69.6 & 30 & $5.32\times 10^{-13}$ & 5.4 &844 & 13\\
$10^{-7}$ & $\boldsymbol{2.04 \times 10^{-8}}$&\textbf{72.0} & \textbf{32} & $5.15\times 10^{-13}$ & 6.0 &844 & 15\\
$10^{-8}$ & $2.65 \times 10^{-9}$&74.0 & 33 & $5.15\times 10^{-13}$ & 6.2 &844 & 17\\
$10^{-9}$ & $2.42 \times 10^{-10}$&75.9 & 34 & $5.15\times 10^{-13}$ & 6.6 &844 & 19\\
$10^{-10}$ & $1.97 \times 10^{-11}$&78.7 & 35 & $5.15\times 10^{-13}$ & 7.0 &844 & 21\\
$10^{-11}$ & $9.61 \times 10^{-12}$&79.6 & 36 & $5.15\times 10^{-13}$ & 7.4 &844 & 23\\
$10^{-12}$ & $0$&89.6 & 39 & $5.15\times 10^{-13}$ & 7.8 &844 & 25\\ \hline \hline
\end{tabular}
}
\end{center}
\end{table}
}

First observe that the solution time increases with reduced optimality tolerance for both algorithms. With lower tolerance, while the barrier algorithm performs more iterations, ALG1 solves more QPs; however, the total number of simplex iterations barely increases. For ALG1 the changes in the value of $t$ are very small between QPs, and the optimal bases of the QPs are thus the same. Therefore, using warm starts, the simplex method is able to find high precision solutions inexpensively.
ALG1 achieves much higher precision an order of magnitude faster than \rev{CPLEX} barrier algorithm. 
For the default tolerance parameters used in our computational experiments, Algorithm~\ref{alg:coordinateDescent} is several orders of magnitude more precise than the barrier algorithm.

\ignore{
	In most settings Algorithm~\ref{alg:coordinateDescent} is more precise than the barrier algorithm with a very low tolerance parameter ($10^{-11}$). Moreover we see that to achieve high precisions the simplex methods require solving more QPs, but the number of simplex iterations does not increase: the changes in the value of $t$ are very small between QPs, and the optimal bases of the QPs are thus the same. Therefore, using warm starts, the simplex methods are able to find high precision solutions inexpensively.
}

\subsubsection*{Effect of the nonlinearity parameter $\Omega$.}
\label{sec:resultsContinuousQ}

We now study the effect of changing the nonlinearity parameter $\Omega$.
%, and we are primarily concerned with comparing the simplex-based algorithms.
Tables \ref{tab:contCard1000} and \ref{tab:contGrid30} show
the total solution time in seconds, the total number of simplex or barrier iterations, and the number of QPs solved in cardinality 
(1000 variables) and path instances (1760 variables), respectively. 
%present results for the continuous instances. 
Each row represents the average over five instances for a rank ($r$) and density($\alpha$) configuration and algorithm used. 
%The tables show the rank and density parameters used to build the matrix $Q$; the solution approach used; and for varying values of the nonlinear coefficient $\Omega$ . 
%For the cardinality instances we use 1,000 variables and for the path instances we use 1,740 variables ($30\times 30$ grid).  
For each parameter choice the fastest algorithm is highlighted in bold. \rev{ Figure~\ref{fig:performanceCont} also shows the total number of instances solved within the given time limit for each instance class.}
%For each type of $Q$ matrix and nonlinear coefficient $\Omega$ we highlight in bold the algorithm configuration that results in the fastest solve times.
%\input{./tables/continuousCard200.tex}
{
%\SingleSpacedXI
\renewcommand\arraystretch{0.75}

\begin{table}[h!]
\caption{The effect of nonlinearity (cardinality instances).} % with $1,000$ variables.}
\begin{center}
\label{tab:contCard1000}
\setlength{\tabcolsep}{2pt}
\SingleSpacedXI
\scalebox{0.9}{
\begin{tabular}{ c c | c | c c c | c c c | c c c | c c c} 
\hline \hline
\multicolumn{2}{c |}{} &\multirow{2}{*}{\texttt{Method}}&\multicolumn{3}{c|}{$\Omega=1$} &\multicolumn{3}{c|}{$\Omega=2$} &\multicolumn{3}{c|}{$\Omega=3$}&\multicolumn{3}{c}{$\Omega=4$} \\
$r$&$\alpha$&&\texttt{time}&\texttt{\#iter} &\texttt{\#QP}&\texttt{time}&\texttt{\#iter} &\texttt{\#QP}&\texttt{time}&\texttt{\#iter} &\texttt{\#QP}&\texttt{time}&\texttt{\#iter} &\texttt{\#QP} \\
\midrule
\multirow{4}{*}{100}&\multirow{4}{*}{0.1}
&ALG1&1.0&22&20&1.1&53&24&1.3&104&29&1.4 & 123&26\\
&&ALG2&\textbf{0.8}&\textbf{41}&\textbf{14}&\textbf{0.9}&\textbf{95}&\textbf{15}&\textbf{0.9}&\textbf{150}&\textbf{15}&\textbf{1.1} &\textbf{219} & \textbf{16}\\
&&BAR&4.6&16&-&4.9&24&-&5.2&26&-& 5.1 & 25 & -\\
&&MOS&2.1&10&-&2.4&11&-&2.2&10&-& 2.4 & 11 & -\\
\midrule
\multirow{4}{*}{100}&\multirow{4}{*}{0.5}
&ALG1&1.1&33&23&1.1&69&24&1.5&144&37&1.6 & 192 & 30\\
&&ALG2&\textbf{0.8}&\textbf{60}&\textbf{14}&\textbf{0.9}&\textbf{125}&\textbf{15}&\textbf{0.9}&\textbf{200}&\textbf{15}& \textbf{1.1} & \textbf{251} & \textbf{16}\\
&&BAR&4.5&21&-&5.1&25&-&5.8&29&-& 5.7 & 29 & -\\
&&MOS&2.2&10&-&2.2&10&-&2.4&10&-& 2.4 & 11 & -\\
\midrule
\multirow{4}{*}{200}&\multirow{4}{*}{0.1}&ALG1&0.9&33&19&1.1&73&25&1.2&110&25& 1.4 & 157 & 26\\
&&ALG2&\textbf{0.8}&\textbf{49}&\textbf{14}&\textbf{0.9}&\textbf{126}&\textbf{14}&\textbf{0.9}&\textbf{172}&\textbf{14}& \textbf{1.1} & \textbf{259} & \textbf{15}\\
&&BAR&4.7&22&-&4.5&22&-&5.1&25&-&5.3 & 27 & -\\
&&MOS&2.4&11&-&2.6&12&-&2.5&11&-&2.4 & 11 & -\\
\midrule
\multirow{4}{*}{200}&\multirow{4}{*}{0.5}&ALG1&1.0&48&22&1.1&99&22&1.2&151&25&1.6 & 218 & 24\\
&&ALG2&\textbf{0.9}&\textbf{94}&\textbf{14}&\textbf{0.9}&\textbf{179}&\textbf{14}&\textbf{1.0}&\textbf{233}&\textbf{15}& \textbf{1.3} & \textbf{326} & \textbf{15}\\
&&BAR&4.4&21&-&4.9&24&-&5.2&26&-&5.8 & 31 & -\\
&&MOS&2.3&10&-&2.4&10&-&2.4&11&-&2.6 & 11 & -\\
\midrule
\multicolumn{2}{c|}{\multirow{4}{*}{avg}}&ALG1&1.0&34&21&1.1&73&24&1.3&127&29&1.5 & 173 & 27\\
&&ALG2&\textbf{0.8}&\textbf{61}&\textbf{14}&\textbf{0.9}&\textbf{131}&
\textbf{15}&\textbf{0.9}&\textbf{189}&\textbf{15} &\textbf{1.2} & \textbf{264} & \textbf{15}\\
&&BAR&4.3&20&-&4.9&24&
-&5.3&27&- &5.5 & 28 & -\\
&&MOS&2.3&10&-&2.4&11&
-&2.4&11&- &2.4 & 11 & -\\
\hline \hline
\end{tabular}
}
\end{center}
\end{table}
}
{
%\SingleSpacedXI
\renewcommand\arraystretch{0.75}
\begin{table}[h!]
\caption{The effect of nonlinearity (path instances).}
\begin{center}
\label{tab:contGrid30}
\setlength{\tabcolsep}{2pt}
\SingleSpacedXI
\scalebox{0.9}{
\begin{tabular}{ c c | c | c c c | c c c | c c c | c c c} 
\hline \hline
\multicolumn{2}{c |}{} &\multirow{2}{*}{\texttt{Method}}&\multicolumn{3}{c|}{$\Omega=1$} &\multicolumn{3}{c|}{$\Omega=2$} &\multicolumn{3}{c|}{$\Omega=3$}&\multicolumn{3}{c}{$\Omega=4$} \\
$r$&$\alpha$&&\texttt{time}&\texttt{\#iter} &\texttt{\#QP}&\texttt{time}&\texttt{\#iter} &\texttt{\#QP}&\texttt{time}&\texttt{\#iter} &\texttt{\#QP}&\texttt{time}&\texttt{\#iter} &\texttt{\#QP} \\
\midrule
\multirow{4}{*}{100}&\multirow{4}{*}{0.1}&ALG1&\textbf{4.9}&\textbf{940}&\textbf{12}&\textbf{7.0}&
\textbf{1,307}&\textbf{16}&8.5&1,505&18& 10.5&  1,756& 21\\
&&ALG2&5.4&1,283&11&\textbf{7.0}&\textbf{1,637}&\textbf{13}&\textbf{8.4}&\textbf{1,865}&\textbf{14}&\textbf{9.7} & \textbf{2,375} & \textbf{13}\\
&&BAR&81.1&26&-&64.3&21&-&55.3&16&-&56.8 & 16 & -\\
&&MOS&19.4&19&-&16.8&17&-&15.6&18&-&15.8 & 19 & -\\
\midrule
\multirow{4}{*}{100}&\multirow{4}{*}{0.5}&ALG1&\textbf{5.2}&\textbf{902}&\textbf{14}&8.2&1,191&21&9.3&1,391&21&11.0 & 1,641 & 21\\
&&ALG2&5.5&1,148&12&\textbf{7.2}&\textbf{1,474}&\textbf{13}&\textbf{8.6}&\textbf{1,772}&\textbf{14}&\textbf{9.6} & \textbf{2,020} & \textbf{14}\\
&&BAR&62.7&19&-&56.0&16&-&57.1&16&-&57.7 & 16 & -\\
&&MOS&17.4&17&-&17.8&18&-&15.7&19&-&14.9 & 17 & -\\
\midrule
\multirow{4}{*}{200}&\multirow{4}{*}{0.1}&ALG1&\textbf{4.9}&\textbf{836}&\textbf{14}&\textbf{6.3}&\textbf{1,053}&\textbf{15}&\textbf{8.3}&\textbf{1,220}&\textbf{18}&14.4&1,429 & 17\\
&&ALG2&\textbf{4.9}&\textbf{932}&\textbf{12}&6.8&1,377&13&8.4&1,671&13&\textbf{12.4}&\textbf{1,833}&\textbf{13}\\
&&BAR&76.8&25&-&60.1&18&-&66.0&20&-&128.3 &21&-\\
&&MOS&16.2&17&-&16.1&18&-&15.5&18&-&15.2 &18&-\\
\midrule
\multirow{4}{*}{200}&\multirow{4}{*}{0.5}&ALG1&\textbf{4.5}&\textbf{858}&\textbf{12}&\textbf{6.2}&\textbf{1,048}&\textbf{15}&\textbf{7.6}&\textbf{1,237}&\textbf{16}&\textbf{12.7} & \textbf{1,387} & \textbf{18}	\\
&&ALG2&4.9&978&12&6.8&1,363&13&8.5&1,626&13&15.9 &1,794 & 14\\
&&BAR&83.1&26&-&72.7&21&-&64.6&18&-&101.2 &16 &-\\
&&MOS&18.1&19&-&16.8&20&-&16.9&20&-&16.2 &19 &-\\
\midrule
\multicolumn{2}{c|}{\multirow{4}{*}{avg}}&ALG1&\textbf{4.9}&\textbf{884}&\textbf{13}&\textbf{6.9}&\textbf{1,150}&
\textbf{17}&\textbf{8.4}&\textbf{1,338}&\textbf{18}&12.1&1,553 &20 \\
&&ALG2&5.2&1,086&12&\textbf{6.9}&\textbf{1,463}&
\textbf{13}&8.5&1,734&13&\textbf{11.9} &\textbf{2,005} &\textbf{14} \\
&&BAR&75.9&24&-&63.2&19&
-&60.7&17&-&86.0&17&- \\
&&MOS&17.8&18&-&16.8&18&
-&15.9&19&-&15.5&18&- \\
\hline \hline
\end{tabular}
}
\end{center}
\end{table}
}

\begin{figure}[h!]
	\centering
	\begin{subfigure}[t]{1.0\columnwidth}
		\centering
		\includegraphics[trim={11cm 6cm 11cm 6cm},clip,width=0.8\textwidth]{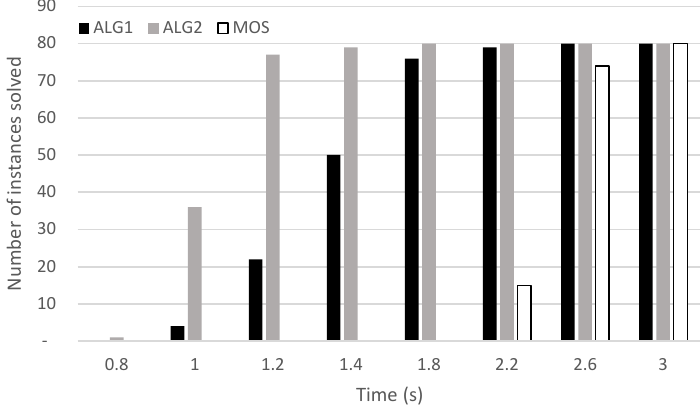}
		\caption{Cardinality instances}
	\end{subfigure}
	
	\begin{subfigure}[t]{1.0\columnwidth}
		\centering
		\includegraphics[trim={11cm 6cm 11cm 6cm},clip,width=0.8\textwidth]{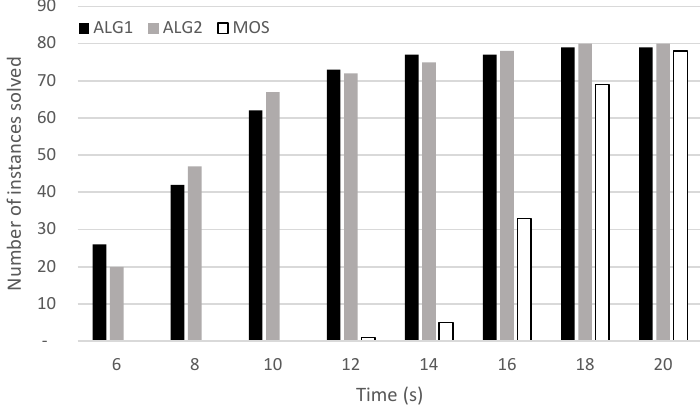}
		\caption{Path instances}
	\end{subfigure}
	\caption{Number of convex instances solved within a time limit for algorithms ALG1, ALG2 and MOS.}
	\label{fig:performanceCont}
\end{figure}

\ignore{
First observe that for both data sets, \rev{CPLEX} barrier algorithm is the slowest: it is 3.5 and 6 times slower than the simplex QP-based methods for the cardinality instances, and is up to 15 times slower for the path instances. 
}
\rev{
	Observe that, compared to CPLEX barrier algorithm, the simplex QP-based methods are  3.5 and 6 times faster for the cardinality instances and up to 15 times faster for the path instances. 
	Additionally, the simplex QP-based methods are two to three times faster than MOSEK barrier algorithm for the cardinality instances, and up to four times faster for the path instances. Figure~\ref{fig:performanceCont} shows that the simplex QP-based methods solve most of the instances well within the time required for MOSEK barrier algorithm to solve the easiest instance}. 

%The simplex QP-based methods are much faster than CPLEX barrier method on all instances.

The barrier algorithm\rev{s} do not appear to be too sensitive to the nonlinearity parameter $\Omega$, whereas the simplex QP-based methods are faster for smaller $\Omega$. 
The number of simplex iterations in ALG1 increases with the nonlinearity parameter $\Omega$. Indeed, the initial problem solved by ALG1 is an LP (corresponding to $\Omega=0$), so as $\Omega$ increases the initial problem becomes a worse approximation, and more work is needed to converge to an optimal solution. 
%On the contrary, the number of simplex iterations of configuration ALG1-1 decreases with $\Omega$: the initial problem ($t_0=1$) corresponds to a small value of the nonlinear term, and is thus more appropriate for large values of $\Omega$. 
Also note that Algorithm~\ref{alg:bisection} requires fewer QPs to be solved, but as a result it benefits less from warm starts (it requires more simplex iterations per QP than ALG1). Indeed, in ALG2 the value of $t$ changes by a larger amount at each iteration (with respect to ALG1), so the objective function of two consecutive QPs changes by a larger amount. \rev{Finally, note that although their runtime is very close, the performance of ALG2 is slightly better than ALG1 overall.}
%These observations hold for other instances sizes as well.

\subsubsection*{Effect of the dimension}
%\label{sec:resultsContinuousDimension}

Table \ref{tab:contCardSizes} presents a comparison of the algorithms for the convex cardinality instances with sizes 400, 800, 1600, and 3200.  Each row represents the average over five instances, as before, generated with parameters $r=200$, $\alpha=0.1$, and $\Omega=2$. 
Additionally, Figure~\ref{fig:improvement} shows the solution time for \rev{ALG1, ALG2 and MOS} as a function of the dimension ($n$).%
% ranging from $200$ to $3,000$. % in multiples of $200$. 
%The speedup factor is computed as $\frac{\text{time}_{B}}{\text{time}_A}$, where  $\text{time}_{B}$ is the time required for the barrier algorithm to solve the problems and $\text{time}_A$ is the time required for a simplex-based algorithm.

{
%\SingleSpacedXI
\renewcommand\arraystretch{1.00}
\begin{table}[h!]
\caption{The effect of dimension (cardinality instances).}
\begin{center}
\label{tab:contCardSizes}
\SingleSpacedXI
\scalebox{0.85}{
\begin{tabular}{  c | c c c | c c c | c c c| c c c  } 
\hline \hline
\multirow{2}{*}{\texttt{Method}}&\multicolumn{3}{c|}{$n=400$} &\multicolumn{3}{c|}{$n=800$}&\multicolumn{3}{c|}{$n=1600$} &\multicolumn{3}{c}{$n=3200$} \\
&\texttt{time}&\texttt{\#iter} &\texttt{\#QP}&\texttt{time}&\texttt{\#iter} &\texttt{\#QP}&\texttt{time}&\texttt{\#iter} &\texttt{\#QP}&\texttt{time}&\texttt{\#iter} &\texttt{\#QP}\\
\hline
ALG1&\textbf{0.2}&\textbf{43}&\textbf{20}&0.6&65&19&2.8&75&25 & 11.7 & 104 & 25\\
ALG2&\textbf{0.2}&\textbf{73}&\textbf{14}&\textbf{0.5}&\textbf{116}&\textbf{14}&\textbf{2.2}&\textbf{129}&\textbf{15}&\textbf{9.1}&\textbf{175}&\textbf{15}\\
BAR&0.3&21&-&2.4&22&-&22.1&27&-&204.9 & 30 & -\\
MOS&\textbf{0.2}&\textbf{9}&\textbf{-}&1.2&11&-&7.3&12&-&50.4 & 12 & -\\
\hline \hline
\end{tabular}
}
\end{center}
\end{table}
}

\begin{figure}[h!]
  \centering
  \includegraphics[width=0.9\textwidth]{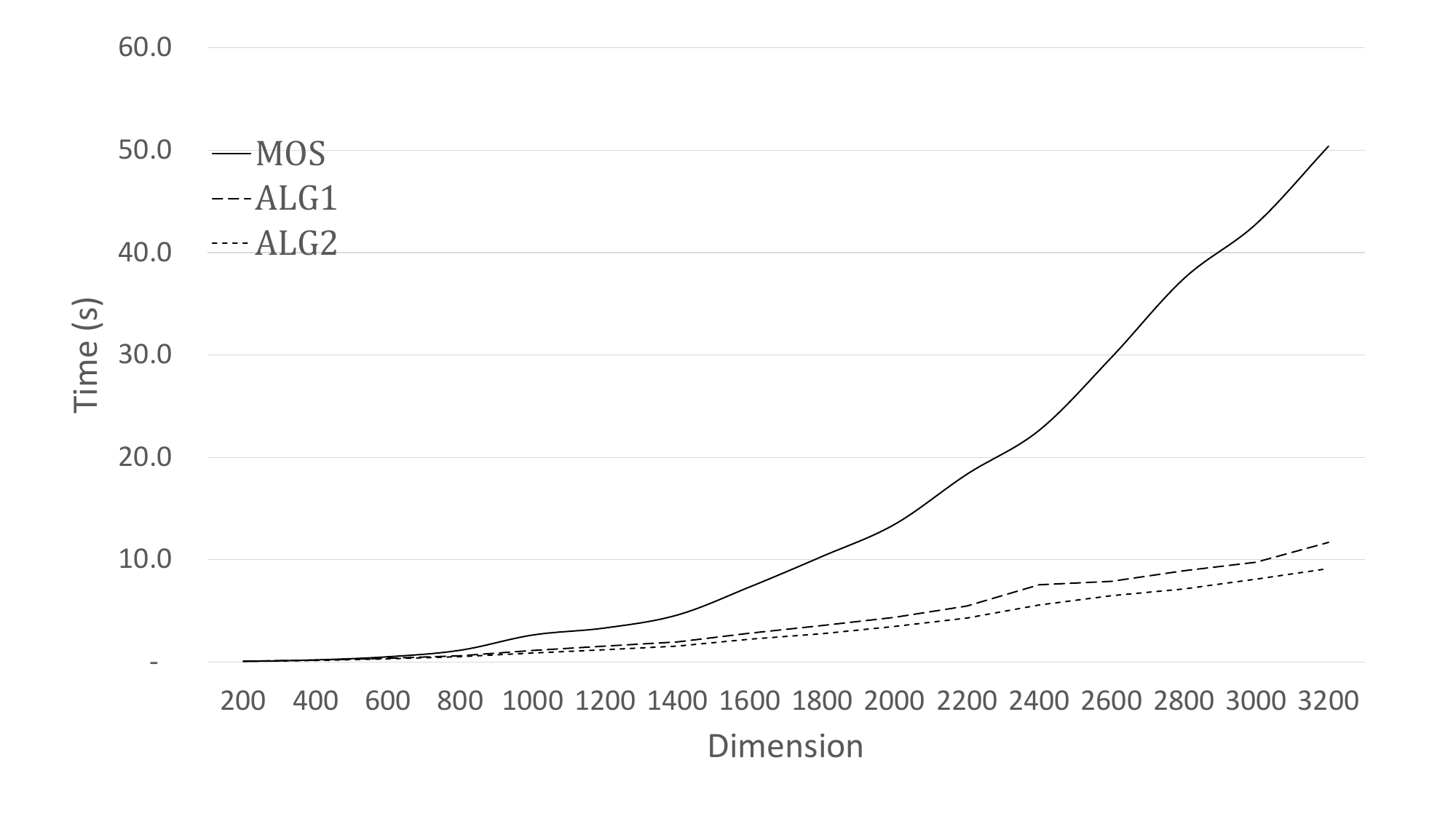}
  \caption{Solution time as a function of dimension.}
  \label{fig:improvement}
\end{figure}

Observe in Table \ref{tab:contCardSizes} that the number of QPs solved with the simplex-based algorithms does not depend on the dimension. The number of simplex iterations, however, increases with the dimension.  For $n=400$ all algorithms perform similarly and the problems are solved very fast. However, as the dimension increases, the simplex-based algorithms outperform the barrier algorithm\rev{s}, often by many factors. For $n=3200$, the fastest simplex-based algorithm ALG2 is more than 20 times faster than \rev{CPLEX} barrier algorithm\rev{, and more than five times faster than MOSEK barrier algorithm}. Similar results are obtained for other parameter choices and for the path instances as well. In summary, the simplex-based algorithms scale better with the dimension, and are faster by orders of magnitude for large instances. \rev{As in Section~\ref{sec:resultsContinuousQ}, ALG2 slightly outperforms ALG1 for the instances considered.}

\subsection{Discrete instances}

%\subsubsection{Discrete problems} 
In this section we describe our experiments with the discrete counterpart \MICO. \rev{To the best of our knowledge, as of version 12.6.2 of CPLEX, there is no documented way to embed a user-defined convex solver such as Algorithm~\ref{alg:coordinateDescent}~or~\ref{alg:bisection} at the nodes of the CPLEX branch-and-bound algorithm.}
%As of version 12.6.2 of CPLEX, it is not possible to employ a user-defined convex solver such as Algorithm~\ref{alg:coordinateDescent} at the nodes of the CPLEX branch-and-bound algorithm.
Therefore, in order to test the proposed approach for \MICO, we implement a rudimentary branch-and-bound algorithm described in Appendix~\ref{sec:branchAndBound}. The  algorithm uses a maximum infeasibility rule for branching, and does not employ presolve, cutting planes, or heuristics. We test the following configurations:
\begin{description}
	\item [BBA1] Branch-and-bound algorithm in Appendix~\ref{sec:branchAndBound} using Algorithm~\ref{alg:coordinateDescent} as the convex solver. %Uses $t_0=1$ at the root node. 
	The first QP at each node (except the root node) is solved with CPLEX dual simplex method using the parent dual feasible basis as a warm start (as mentioned in Section~\ref{sec:warmStarts}) and all other QPs are solved with CPLEX primal simplex method using the basis from the parent node QP as a warm start.
	\rev{\item [BBA2] Branch-and-bound algorithm in Appendix~\ref{sec:branchAndBound} using Algorithm~\ref{alg:bisection} as the convex solver. %Uses $t_0=1$ at the root node. 
	Algorithm~\ref{alg:bisection} resulted in the best performance in the continuous instances; however, unlike Algorithm~\ref{alg:coordinateDescent}, it cannot be naturally warm-started. Thus, in this configuration, each convex subproblem is solved without exploiting the solution from the parent node.}
	\item [BBBR] Branch-and-bound algorithm in Appendix~\ref{sec:branchAndBound}, using CPLEX barrier algorithm as the convex solver. This configuration does not use warm starts.
	\item [CXBR] CPLEX branch-and-bound algorithm with barrier solver, setting the branching rule to maximum infeasibility, the node selection rule to best bound, and disabling presolve, cuts and heuristics. In this setting CPLEX branch-and-bound algorithm is as close as possible to our branch-and-bound algorithm. 
	\item [CXLP] CPLEX branch-and-bound algorithm with LP outer approximations, setting the branching rule to maximum infeasibility, the node selection rule to best bound, and disabling presolve, cuts and heuristics. In this setting CPLEX branch-and-bound algorithm is as close as possible to our branch-and-bound algorithm. 
	\item [CXLPE] CPLEX branch-and-bound algorithm with LP outer approximations, setting the branching rule to maximum infeasibility, the node selection rule to best bound, and disabling cuts and heuristic. Since presolve is activated, CPLEX uses extended formulations described in \cite{Vielma2015}. Besides presolve, all other parameters  are set as in CXLP.%\todo{what are the other settings?}
	\item [CXD] CPLEX default branch-and-bound algorithm with LP outer approximations. \rev{This algorithm utilizes all sophisticated features of CPLEX, such as presolver, cutting planes, heuristics, advanced branching and node selection rules.}
\end{description}
The time limit is set to two hours for each algorithm.
\ignore{
 \rev{Observe that, in both cases, the feasible regions correspond to the extreme points of an integral polytope. Therefore, many of techniques for strengthening formulations (e.g., cutting planes) used by CXD are not very effective for these instances. As a consequence, although CXD uses many additional features (e.g., heuristics and sophisticated branching rules), it is still comparable to the other formulations.}
}

Table \ref{tab:discCard200} presents the results for discrete cardinality instances with 200 variables and Table~\ref{tab:discGrid30} for the  discrete path instances with 1,740 variables ($30\times 30$ grid). Each row represents the average over five instances with varying rank and density parameters, and algorithm. The tables show the solution time in seconds,  the number of nodes explored in the branch-and-bound tree, the end gap after two hours as percentage, and the number of instances that are solved to optimality for varying values of $\Omega$. For each instance class we highlight in bold the algorithm with the best performance. \rev{Figure~\ref{fig:performanceDisc} shows, for each instance class, the total number of instances solved within given time limits for BBA1, CXLPE and CXD.}
% that result in the best performance in terms of both time and end gaps (when such a configuration exists). 

{
%\SingleSpacedXI
\renewcommand\arraystretch{1.00}
\setlength{\tabcolsep}{0.5pt}
\begin{table}[h!]
\caption{Comparison for discrete cardinality instances.}
\begin{center}
\label{tab:discCard200}
\SingleSpacedXI
\scalebox{0.8}{
\begin{tabular}{ c c | c | c c c c | c c c c| c c c c | c c c c } 
\hline \hline
\multicolumn{2}{c |}{} &\multirow{2}{*}{\texttt{Method}}&\multicolumn{4}{c|}{$\Omega=1$} &\multicolumn{4}{c|}{$\Omega=2$} &\multicolumn{4}{c|}{$\Omega=3$}&\multicolumn{4}{c}{$\Omega=4$}\\
$r$&$\alpha$&&\texttt{time}&\texttt{nodes} &\texttt{egap}&\texttt{\#s}&\texttt{time}&\texttt{nodes} &\texttt{egap}&\texttt{\#s}&\texttt{time}&\texttt{nodes} &\texttt{egap}&\texttt{\#s} &\texttt{time}&\texttt{nodes} &\texttt{egap}&\texttt{\#s}\\
\hline
\multirow{7}{*}{100}&\multirow{7}{*}{0.1}&BBA1&\textbf{1}&\textbf{156}&\textbf{0.0}&\textbf{5}&\textbf{29}&\textbf{3,271}
&\textbf{0.0}&\textbf{5}&685&68,318&0.0&5&\textbf{3,644} & \textbf{272,527} & \textbf{0.1} & \textbf{4}\\
&&BBA2&3&156&0.0&5&83&3,271&0.0&5&1,823&68,317&0.0&5&6,218 & 172,489 & 0.3 & 2\\
&&BBBR&16&156&0.0&5&349&3,270&0.0&5&4,664&43,695&0.1&3& 7,200 & 23,324 & 1.2 & 0 \\
&&CXBR&35&276&0.0&5&513&3,497&0.0&5&5,260&32,782&0.2&2&7,200 & 17,169 & 1.3 & 0\\
&&CXLP&34&9,562&0.0&5&7,200&209,576&0.7&0&7,200&244,911&2.2&0 & 7,200 & 265,485 & 3.8 & 0\\
&&CXLPE&2&374&0.0&5&91&7,640&0.0&5&2,788&111,293&0.0&5&6,983 & 191,065 & 0.6 & 1\\
&&CXD&4&368&0.0&5&42&5,152&0.0&5&\textbf{640}&\textbf{58,076}&\textbf{0.0}&\textbf{5}&3,778 & 183,816 & 0.1 & 4\\
\midrule
\multirow{7}{*}{100}&\multirow{7}{*}{0.5}&BBA1&\textbf{1}&\textbf{87}&\textbf{0.0}&\textbf{5}&\textbf{59}&\textbf{6,274}&\textbf{0.0}&
\textbf{5}&\textbf{1,469}&\textbf{140,874}&\textbf{0.0}&\textbf{5}&\textbf{6,328} & \textbf{447,989} & \textbf{0.4} & \textbf{2}\\
&&BBA2&2&87&0.0&5&160&6,274&0.0&5&4,024&139,154&0.0&4& 7,200 & 223,921 & 0.7 & 0\\
&&BBBR&10&87&0.0&5&686&6,274&0.0&5&6,134&56,394&0.3&1 & 7,200 & 21,111 & 1.7 & 0\\
&&CXBR&24&183&0.0&5&1,027&6,734&0.0&5&6,399&39,710&0.4&1 & 7,200 & 20,101 & 2.0 & 0\\
&&CXLP&294&26,957&0.0&5&7,200&229,641&0.8&0&7,200&263,810&2.3&0 & 7,200 & 244,863 & 4.7 & 0\\
&&CXLPE&2&349&0.0&5&218&14,737&0.0&5&5,116&215,292&0.1&2&7,200 & 170,710 & 1.1 & 0\\
&&CXD&3&373&0.0&5&164&16,070&0.0&5&3,643&245,251&0.0&4& 7,042 & 336,814 & 0.8 & 1\\
\midrule
\multirow{7}{*}{200}&\multirow{7}{*}{0.1}&BBA1&\textbf{1}&\textbf{247}&\textbf{0.0}&\textbf{5}&\textbf{23}&\textbf{3,259}&\textbf{0.0}&\textbf{5}&\textbf{637}
&\textbf{55,248}&\textbf{0.0}&\textbf{5}&\textbf{3,761} & \textbf{344,662} & \textbf{0.2} & \textbf{4} \\
&&BBA2&5&247&0.0&5&79&3,259&0.0&5&2,083&55,248&0.0&5&6,975 & 242,548 & 0.4 & 2\\
&&BBBR&24&247&0.0&5&321&3,259&0.0&5&4,573&39,647&0.1&3 & 7,200 & 17,017 & 1.4 & 0\\
&&CXBR&52&460&0.0&5&540&3,711&0.0&5&5,295&34,090&0.2&2 & 7,200 & 13,490 & 1.5 & 0\\
&&CXLP&221&17,205&0.0&5&7,200&208,874&0.6&0&7,200&230,304&2.0&0 & 7,200 & 186,490 & 4.2 & 0\\
&&CXLPE&4&473&0.0&5&139&6,064&0.0&5&4,073&111,205&0.1&3& 7,200 & 158,866 & 0.9 & 0\\
&&CXD&5&360&0.0&5&60&6,413&0.0&5&1,410&67,577&0.0&5&7,044 & 349,653 & 0.5 & 1\\
\midrule
\multirow{7}{*}{200}&\multirow{7}{*}{0.5}&BBA1&\textbf{4}&\textbf{674}&\textbf{0.0}&\textbf{5}&
\textbf{194}&\textbf{24,636}&\textbf{0.0}&\textbf{5}&\textbf{1,674}&\textbf{156,632}&\textbf{0.0}&\textbf{5}&\textbf{5,778} & \textbf{526,215} & \textbf{0.3} & \textbf{2}\\
&&BBA2&15&674&0.0&5&633&24,635&0.0&5&3,446&98,028&0.1&4 & 7,200 & 259,040 & 0.6 & 0\\
&&BBBR&77&674&0.0&5&2,106&17,743&0.0&4&5,590&47,725&0.2&2&7,200 & 20,422 & 1.5 & 0\\
&&CXBR&104&680&0.0&5&2,452&15,816&0.0&4&6,127&38,973&0.3 & 1 & 7,200 & 14,955 & 1.7 & 0\\
&&CXLP&3,514&120,007&0.1&4&7,200&212,082&1.0&0&7,200&240,445&2.3&0 & 7,200 & 195,841 & 4.8 & 0\\
&&CXLPE&21&1,461&0.0&5&1,739&61,593&0.0&4&5,435&163,105&0.2&2& 7200 & 197,287 & 1.1 & 0\\
&&CXD&18&1,612&0.0&5&1,211&75,098&0.0&5&5,017&245,412&0.2&2&7,200 & 319,645 & 1.0 & 0\\
\midrule
\multicolumn{2}{c|}{\multirow{7}{*}{avg}}&BBA1&\textbf{2}&\textbf{291}&\textbf{0.0}&\textbf{20}&\textbf{76}&\textbf{9,360}&
\textbf{0.0}&\textbf{20}&\textbf{1,116}&\textbf{105,268}&\textbf{0.0}&\textbf{20}&\textbf{4,878}&\textbf{397,848}&\textbf{0.3}&\textbf{12} \\
&&BBA2&6&291&0.0&20&239&9,360&
0.0&20&2,844&90,187&0.0&18 & 6,898 & 224,500 & 0.5 &  4\\
&&BBBR&32&291&0.0&20&865&7,637&
0.0&19&5,240&46,865&0.2&9 & 7,200 & 20,469 & 1.4 & 0\\
&&CXBR&54&400&0.0&20&1,133&7,440&
0.0&19&5,770&36,389&0.3&6 & 7,200 & 16,429 & 1.6 & 0\\
&&CXLP&1,016&43,433&0.0&19&7,200&215,043&0.8&0&7,200&244,867&2.2&0 & 7,200 & 223,170 & 4.4 & 0\\
&&CXLPE&7&664&0.0&20&547&20,800&
0.0&19&4,353&139,632&0.1&12 & 7,146 & 179,482 & 0.9 & 1 \\
&&CXD&7&678&0.0&20&369&25,683&0.0&20&2,677&151,588&0.1&16 & 6,267 & 297,482 & 0.6 & 6\\
\hline \hline
\end{tabular}
}
\end{center}
\end{table}
}
{
%\SingleSpacedXI
\renewcommand\arraystretch{1.00}
\setlength{\tabcolsep}{0.5pt}
\begin{table}[h!]
\caption{Comparison for discrete path instances.}
\begin{center}
\label{tab:discGrid30}
\SingleSpacedXI
\scalebox{0.8}{
\begin{tabular}{ c c | c | c c c c | c c c c | c c c c | c c c c} 
\hline \hline
\multicolumn{2}{c |}{\textbf{}} &\multirow{2}{*}{\texttt{Method}}&\multicolumn{4}{c|}{$\Omega=1$} &\multicolumn{4}{c|}{$\Omega=2$} &\multicolumn{4}{c|}{$\Omega=3$}&\multicolumn{4}{c}{$\Omega=4$} \\
$r$&$\alpha$&&\texttt{time}&\texttt{nodes} &\texttt{egap}&\texttt{\#s}&\texttt{time}&\texttt{nodes} &\texttt{egap}&\texttt{\#s}&\texttt{time}&\texttt{nodes} &\texttt{egap}&\texttt{\#s}&\texttt{time}&\texttt{nodes} &\texttt{egap}&\texttt{\#s} \\
\hline
\multirow{7}{*}{100}&\multirow{7}{*}{0.1}&BBA1&\textbf{287}&\textbf{145}&\textbf{0.0}
&\textbf{5}&\textbf{4,511}&\textbf{1,774}&\textbf{0.0}&\textbf{5}&\textbf{7,200}&\textbf{2,720}&\textbf{5.9}&\textbf{0}&\textbf{7,200}&\textbf{2,360}&\textbf{13.9}&\textbf{0}\\
&&BBA2&619&142&0.0&5&6,871&1,105&1.1&1&7,200&1,278&7.9&0 & 7,200 & 974 & 17.2 & 0\\
&&BBBR&3,577&91&0.2&4&7,200&184&4.9&0&7,200&236&11.7&0&7,200 & 61 & 37.2 & 0\\
&&CXBR&7,200&67&20.8&0&7,200&79&$\infty$&0&7,200&129&$\infty$&0&7,200 & 13 & $\infty$ & 0\\
&&CXLP&533&2,428&0.0&5&7,200&24,776&4.5&0&7,200&20,099&15.0&0& 7,200 & 8,971& 30.2 & 0\\
&&CXLPE&802&315&0.0&5&6,726&1,967&2.9&1&7,200&2,585&23.5&0&7,200 & 2,377 & 45.1 & 0\\
&&CXD&1,466&164&0.0&5&4,655&1,176&0.0&5&7,200&2,428&7.4&0&7,200 & 2,047 & 16.6 & 0\\
\midrule
\multirow{7}{*}{100}&\multirow{7}{*}{0.5}&BBA1&\textbf{625}&\textbf{353}&\textbf{0.0}&\textbf{5}&\textbf{5,424}&\textbf{1,904}&\textbf{0.6}&\textbf{2}&\textbf{6,512}&\textbf{2,725}&\textbf{4.8}&\textbf{1}& \textbf{7,200} & \textbf{2,833} & \textbf{12.5} & \textbf{0}\\
&&BBA2&1,457&362&0.0&5&6,286&971&1.6&1&7,200&1,369&7.3&0&7,200 & 1,296 & 16.2 & 0\\
&&BBBR&6,071&134&0.7&2&7,200&175&4.9&0&7,200&162&11.5&0& 7,200 & 28 & $\infty$ & 0\\
&&CXBR&7,200&24&$\infty$&0&7,200&56&$\infty$&0&7,200&70&$\infty$&0& 7,200 & 13 & $\infty$ & 0\\
&&CXLP&1,132&6,187&0.0&5&7,200&23,671&4.4&0&7,200&16,851&12.8&0 & 7,200 & 8,180 & 23.8 & 0\\
&&CXLPE&967&607&0.0&5&6,420&2,077&3.4&1&7,200&3,070&16.2&0&7,200 & 3,036 & 42.3 & 0\\
&&CXD&1,645&267&0.0&5&\textbf{6,421}&\textbf{1,931}&\textbf{0.5}&\textbf{3}&7,200&2,659&5.6&0&7,200 & 4,166 & 12.8 & 0\\
\midrule
\multirow{7}{*}{200}&\multirow{7}{*}{0.1}&BBA1&\textbf{155}&\textbf{77}&\textbf{0.0}&\textbf{5}&\textbf{2,392}&\textbf{1,075}&\textbf{0.0}&\textbf{5}&\textbf{6,434}&\textbf{3,380}&\textbf{2.6}&\textbf{1}&\textbf{7,200} &\textbf{ 3,245} & \textbf{10.0} & \textbf{0}\\
&&BBA2&306&79&0.0&5&4,324&823&0.3&3&7,152&1,469&4.6&1 & 7,200 & 1,195 & 12.7 & 0\\
&&BBBR&3,245&76&0.0&5&7,200&171&2.4&0&7,200&180&10.5&0 & 7,200 & 33 & $\infty$ & 0\\
&&CXBR&7,200&34&$\infty$&0&7,200&45&$\infty$&0&7,200&68&$\infty$&0& 7,200 & 11 & $\infty$&0\\
&&CXLP&436&1,548&0.0&5&7,200&30,265&2.9&0&7,200&20,579&12.4&0 & 7,200 & 7,274 & 26.7 & 0\\
&&CXLPE&524&188&0.0&5&5,156&1,437&1.4&3&7,200&2,420&17.7&0&7,200 & 2,157 & 46.7 & 0\\
&&CXD&2,059&106&0.0&5&5,715&1,251&0.5&4&7,200&2,568&4.1&0&7,200 & 1,996 & 12.9 & 0\\
\midrule
\multirow{7}{*}{200}&\multirow{7}{*}{0.5}&BBA1&\textbf{321}&\textbf{196}&\textbf{0.0}&\textbf{5}&\textbf{3,953}&\textbf{2,286}&\textbf{0.3}&\textbf{4}&
\textbf{7,200}&\textbf{3,194}&\textbf{4.3}&\textbf{0}&\textbf{7,200}&\textbf{2,486}&\textbf{13.2}&\textbf{0}\\
&&BBA2&808&201&0.0&5&5,517&1,204&0.9&3&7,200&1,284&6.5&0&7,200 & 1,009 & 16.6 & 0\\
&&BBBR&4,826&113&0.2&3&7,200&173&3.9&0&7,200&176&12.7&0& 7,200 & 54 & 43.6 & 0\\
&&CXBR&7,200&20&$\infty$&0&7,200&51&$\infty$&0&7,200&89&$\infty$&0& 7,200 & 10 & $\infty$& 0\\
&&CXLP&859&4,989&0.0&5&7,200&28,007&4.4&0&7,200&18,873&13.3&0& 7,200 & 10,313 & 28.9 & 0\\
&&CXLPE&1,046&399&0.0&5&5,948&2,212&1.8&2&7,200&2,305&21.2&0&7,200 & 2,048 & 51.2 & 0\\
&&CXD&2,281&177&0.0&5&4,975&1,873&0.4&4&7,200&2,233&6.9&0&7,200 & 1,325 & 16.0 & 0\\
\midrule
\multicolumn{2}{c|}{\multirow{7}{*}{avg}}&BBA1&\textbf{347}&\textbf{193}&\textbf{0.0}&\textbf{20}&\textbf{4,070}&\textbf{1,760}&
\textbf{0.2}&\textbf{16}&\textbf{6,837}&\textbf{3,005}&\textbf{4.4}&\textbf{2} &\textbf{7,200} & \textbf{2,731} & \textbf{12.4}& \textbf{0}\\
&&BBA2&798&196&0.0&20&5,750&1,026&
1.0&8&7,189&1,350&6.6&1& 7,200 & 1,119 & 15.7 & 0\\
&&BBBR&4,430&103&0.3&14&7,200&176&
4.0&0&7,200&189&11.6&0 &7,200 & 44 & $\infty$ & 0\\
&&CXBR&7,200&36&$\infty$&0&7,200&58&
$\infty$&0&7,200&89&$\infty$&0&7,200 & 12 & $\infty$ & 0 \\
&&CXLP&740&3,788&0.0&20&7,200&26,680&
4.1&0&7,200&19,101&13.4&0 & 7,200 & 8,684 & 27.4 & 0\\
&&CXLPE&835&377&0.0&20&6,063&1,923&
2.4&7&7,200&2,595&19.7&0 &7,200 & 2,404 & 46.3 & 0\\
&&CXD&1,863&178&0.0&20&5,441&1,558&
0.3&16&7,200&2,472&6.0&0 & 7,200 & 2,383 & 14.6 & 0 \\
\hline \hline
\end{tabular}
}
\end{center}
\end{table}
}

\begin{figure}[h!]
	\centering
	\begin{subfigure}[t]{1.0\columnwidth}
		\centering
		\includegraphics[trim={11cm 6cm 11cm 6cm},clip,width=0.8\textwidth]{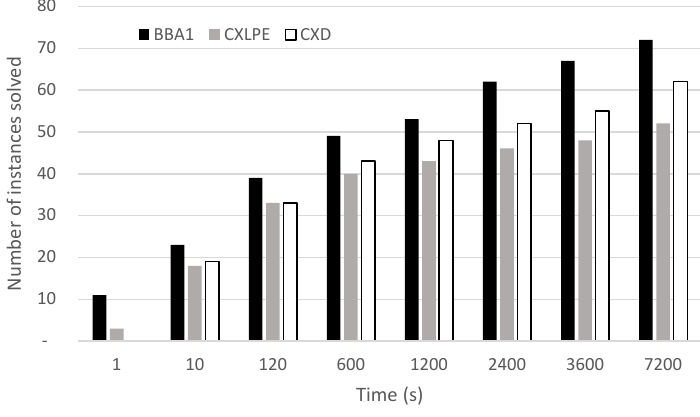}
		\caption{Cardinality instances}
	\end{subfigure}
	
	\begin{subfigure}[t]{1.0\columnwidth}
		\centering
		\includegraphics[trim={11cm 6cm 11cm 6cm},clip,width=0.8\textwidth]{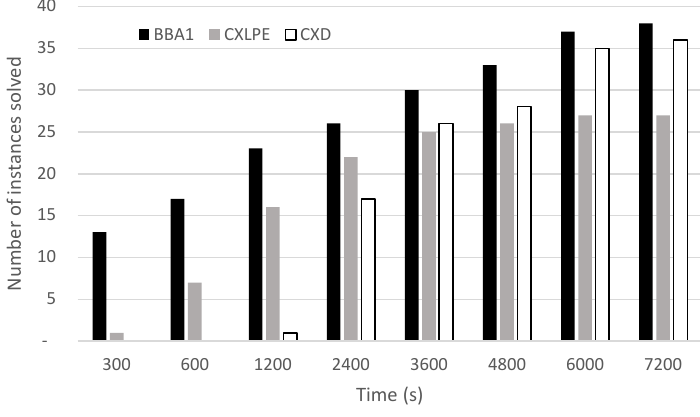}
		\caption{Path instances}
	\end{subfigure}
	\caption{Number of discrete instances solved within a time limit.}
	\label{fig:performanceDisc}
\end{figure}

\ignore{
We first give some general comments in Section~\ref{sec:discGeneralComments}, then in Section~\ref{sec:performanceBBA1} we comment on the performance of configuration BBA1, and finally in Section~\ref{sec:discWarmStarts} we study the impact of warm starts.
}

%\subsubsection{General comments}
%\label{sec:discGeneralComments} 
First of all, observe that the difficulty of the instances increases considerably for higher values of $\Omega$ due to higher integrality gap. The problems corresponding to high values of the density parameter $\alpha$ are also more challenging.

\subsubsection*{Performance of CPLEX branch-and-bound}
%\label{sec:performanceBBA1}

Among CPLEX branch-and-bound algorithms, CXD is the best choice when $\Omega\geq 2$. Configuration CXD is much more sophisticated than the other configurations, so a better performance is expected. However, note that for $\Omega=1$ configuration CXD is not necessarily the best. In particular in the path instances (Table~\ref{tab:discGrid30}) CXLP and CXLPE are 2.3 times faster than CXD. This result suggests that in simple instances the additional features used by CXD (e.g. cutting planes and heuristics) may be hurting the performance.

The extended formulations result in much stronger relaxations in LP based branch-and-bound and, consequently, the number of branch-and-bound nodes required with CXLPE is only a small fraction of the number of nodes required with CXLP. However, CXLPE requires more time to solve each branch-and-bound node, due to the higher number of variables and the additional effort needed to refine the LP outer approximations. For the cardinality instances, CXLPE is definitely the better choice and is faster by orders of magnitude. For the path instances, however, CXLP is not necessarily inferior: when $\Omega=1$ CXLP is competitive with CXLPE, and when $\Omega=3$ CXLP performs better.

The barrier-based branch-and-bound CXBR, in general, performs poorly. For the cardinality instances, it outperforms CXLP but is slower than the other algorithms. For the path instances it has the worst performance, often struggling to find even a single feasible solution (resulting in infinite end gaps).

\subsubsection*{Performance of BBA1}
%\label{sec:performanceBBA1}

Note that BBA1, \rev{BBA2} and BBBR are very simple and differ only by the convex node solver. BBA1 is faster than BBBR by an order of magnitude. \rev{BBA1 is also two to three times faster than BBA2, despite the fact that Algorithm~\ref{alg:bisection} is faster for convex problems. This improvement is due to the warm start capabilities of Algorithm~\ref{alg:coordinateDescent}.} BBA1 is considerably faster than the simplest CPLEX branch-and-bound algorithms CXBR and CXLP.

We see that BBA1 \rev{consistently} outperforms CXLPE (which uses presolve and extended formulations)\rev{, often by many factors. In fact, BBA1 resulted in better performance (faster solution times or lower end gaps) than CXLPE in every instance}. Observe that in the cardinality instances with $\Omega=1,2$ and path instances with $\Omega=1$,  BBA1 requires half the number of nodes (or less) compared to CXLPE to solve the instances to optimality (since the relaxations solved at each node are stronger), which translates into faster overall solution times. In the more difficult instances BBA1 is able to solve more instances to optimality, and the end gaps are smaller. 

Despite the fact that BBA1 is a rudimentary branch-and-bound implementation, it is faster than default CPLEX in most of the cases. \rev{Indeed, BBA1 outperforms CXD in 143 out of 160 instances tested. Figure~\ref{fig:performanceDisc} clearly shows that BBA1 solves more instances faster compared to CXLPE and CXD.} 
%\todo{The figures are very clear. No need to brag further.}

\ignore{
	%%% dropped this paragraph %%%%
Moreover, in the instances where CXD is better the difference between the algorithms is small (less than 10\% difference in solution times\rev{, on average}), while in the other instances BBA1 is often faster by many factors. \rev{For example, for the path instances,  Figure~\ref{fig:performanceDisc} shows that BBA1 is able to solve close to 25 instances in 20 minutes or less, while CXD is able to solve only one instance in the same time limit, and requires close to one hour to solve 25 instances.}
\ignore{We observe that CXD is comparatively better for the instances with a low factor rank ($r=100$), and BBA1 is comparatively better for the instances with a high factor rank ($r=200$).
}
}

\subsubsection*{Warm starts}
%\label{sec:discWarmStarts}

%Algorithm BBA1 is faster than BBBR in part due to a faster convex solver (as observed in Section~\ref{sec:resultsContinuous}), and in part due to node warm starts. 
To quantify the impact of warm starts, we plot in Figure~\ref{fig:timePerNode} the \emph{time per node} (computed as solution time divided by the number of branch-and-bound nodes) for BBA1, \rev{BBA2,} BBBR and CXLPE, and also plot the solution time for the corresponding convex instances with solvers ALG1\rev{, ALG2,} and BAR\footnote{The time per node is similar for all combinations of parameters $\Omega$, $r$ and $\alpha$. We plot the average for all instances with $\Omega=2$.}.

\begin{figure}[h!]
\centering
\begin{subfigure}[t]{0.5\columnwidth}
  \centering
  \includegraphics[width=1\textwidth]{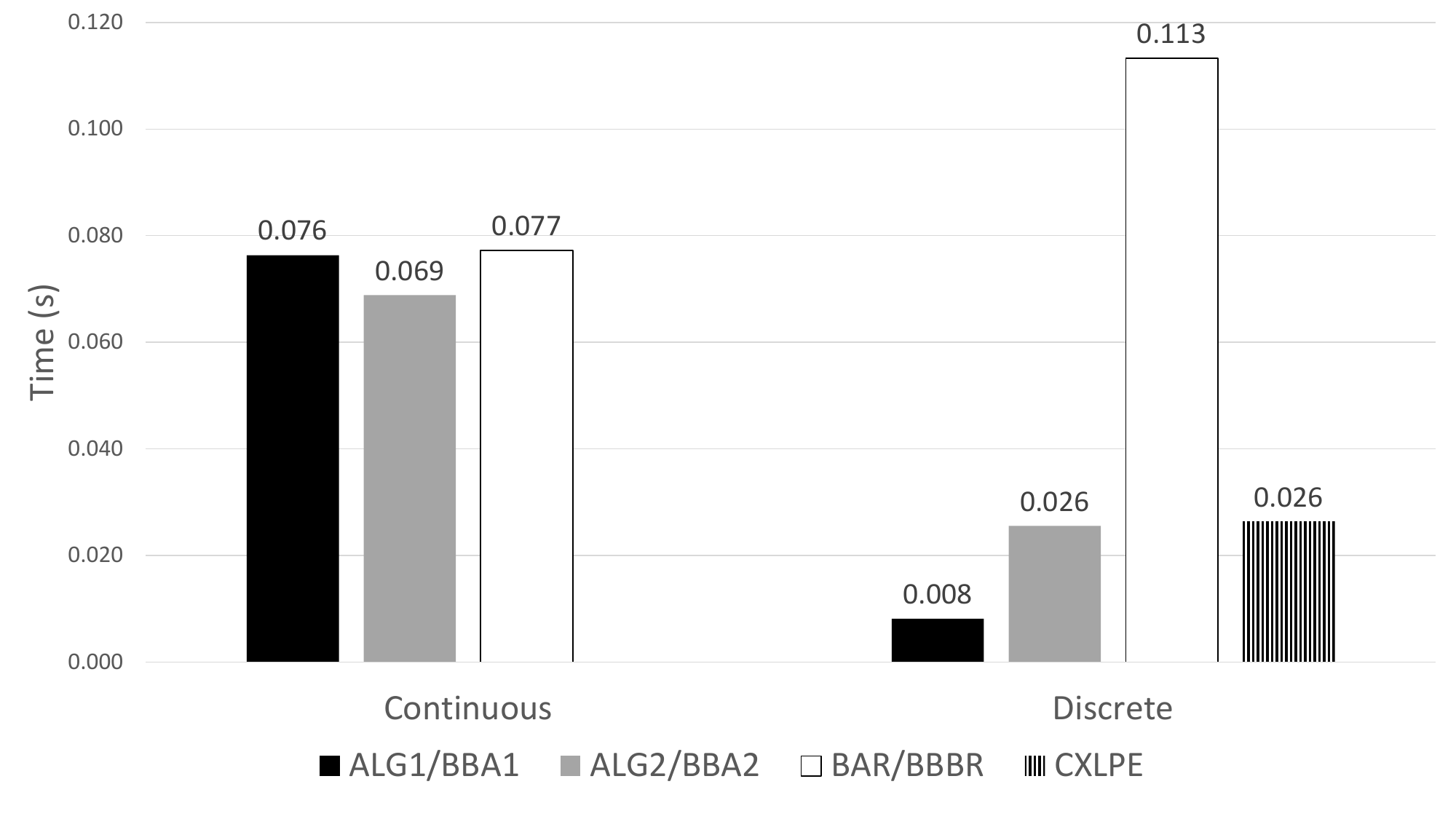}
  \caption{Cardinality instances}
\end{subfigure}
~
\begin{subfigure}[t]{0.5\columnwidth}
  \centering
  \includegraphics[width=1\textwidth]{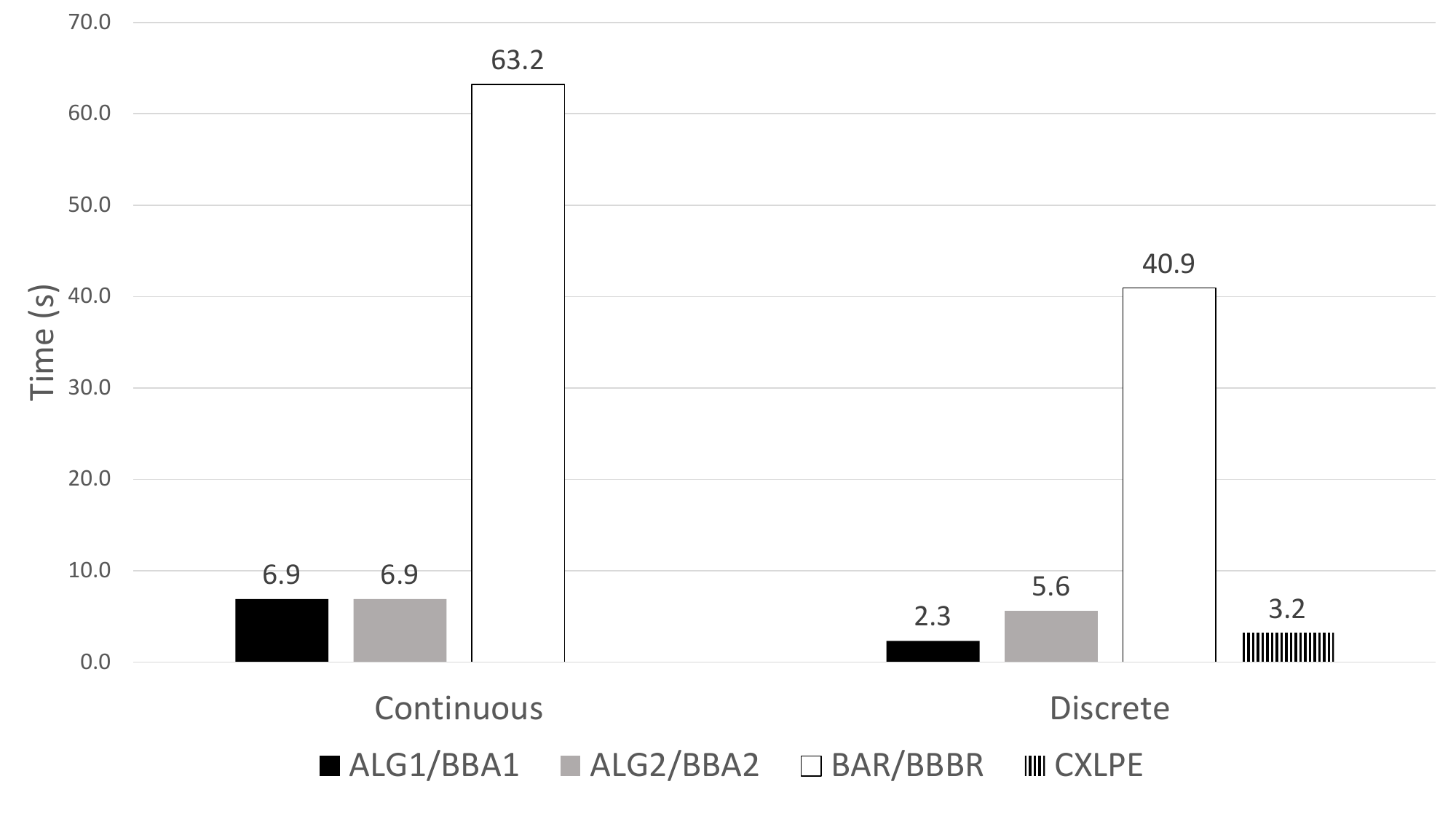}
  \caption{Path instances}
\end{subfigure}
\caption{Time per node.}
\label{fig:timePerNode}
\end{figure}

For the small cardinality instances with 200 variables, \rev{all three algorithms perform similarly for the convex instances}; however, \rev{Algorithm~\ref{alg:coordinateDescent}} is 15 times faster than barrier \rev{and more than three times faster than Algorithm~\ref{alg:bisection}} when used in branch-and-bound due to the node warm starts from dual feasible solutions. %Computational experiments for cardinality instances with other dimensions similarly result in BBA1 being 15 times faster than it would be in continuous instances. 
For the larger path instances with 1,740 variables, Algorithm~\ref{alg:coordinateDescent} \rev{is again close to three times faster than Algorithm~\ref{alg:bisection} and almost 20 times faster than barrier in discrete instances.}
% - similar results are obtained in path instances with other dimensions). 
Finally, observe that the solve time per node for BBA1 is smaller compared to CXLPE: the proposed simplex-based algorithm is thus as effective as the simplex method for extended formulations in exploiting warm starts. 
Moreover, it solves the nonlinear convex relaxations at each node to optimality, whereas CXLPE solves its LP relaxation.
The improved lower bounds lead to significantly small search trees.

We conclude that Algorithm~\ref{alg:coordinateDescent} is indeed suitable for branch-and-bound algorithms since it benefits from node warms starts from the parent nodes, resulting in a significant improvement in solution times.  

\ignore{

\section{Extensions}
\label{sec:extensions}

We discuss in this section how to extend the algorithms of Section~\ref{sec:algorithms} to SOCPs with linear objective and a single conic quadratic constraint using a Lagrangean relaxation. We have that
\begin{align*}
&\min_{x\in X }\left\{c'x:d'x+\Omega\sqrt{x'Qx}\leq b_0\right\}\\
=&\min_{x\in X,s\geq 0 }\left\{c'x:d'x+\Omega s\leq b_0, \sqrt{x'Qx}\leq s\right\}\\
=&\max_{\lambda\geq 0 }\min_{x\in X,s\geq 0 }\left\{c'x+\lambda \sqrt{x'Qx}-\lambda s:d'x+\Omega s\leq b_0\right\}\\
=&\max_{\lambda\geq 0 }\min_{x\in X,s,t\geq 0 }\left\{c'x+ \frac{\lambda}{2t}x'Qx+\frac{\lambda t}{2}-\lambda s:d'x+\Omega s\leq b_0\right\}\\
=&\max_{\lambda\geq 0}h(\lambda).
\end{align*}
The function $h$ is a concave univariate function, and the optimal $\lambda^*$ can be found using bisection search. Evaluating function $h$ for a fixed $\lambda$ requires solving a problem of the form \PO, which can be done using the algorithms of Section~\ref{sec:algorithms}. Moreover, each evaluation $h(\lambda)$ can be warm started using the optimal basis from the previous evaluation.

We tested a simple version of this Lagrangean relaxation approach, using Algorithm~\ref{alg:coordinateDescent} to solve the QPs, but our results were not as good as those reported in Section~\ref{sec:computational}. In the continuous instances the algorithm was slightly worse than CPLEX barrier algorithm (between 10\% and 20\% slower); in discrete instances, using a branch-and-bound algorithm based on Lagrangean relaxations, it was twice as slow as CPLEX LP branch-and-bound with extended formulations. Nevertheless, using the Lagrangean relaxation may be useful in problems where only a lower bound is required (i.e., solving for a fixed $\lambda$ instead of searching for $\lambda^*$), or in problems where the QPs are particularly easy to solve.

}

\section{Conclusions}
\label{sec:conclusions}

We consider minimization problems with a conic quadratic objective and linear constraints, which are natural generalizations of linear programming and quadratic programming. Using the perspective function we reformulate the objective and propose simplex QP-based algorithms that solve a quadratic program at each iteration. Computational experiments indicate that the proposed algorithms are faster than interior point methods by orders of magnitude, scale better with the dimension of the problem, return higher precision solutions, and, most importantly, are amenable to warm starts. Therefore, they can be embedded in branch-and-bound algorithms quite effectively.

\section*{Acknowledgement}
This research is supported, in part,
by grant FA9550-10-1-0168 from the Office of the Assistant Secretary of Defense for Research and Engineering.

\bibliographystyle{plainnat}
\bibliography{Bibliography}

\begin{thebibliography}{42}
\providecommand{\natexlab}[1]{#1}
\providecommand{\url}[1]{\texttt{#1}}
\expandafter\ifx\csname urlstyle\endcsname\relax
  \providecommand{\doi}[1]{doi: #1}\else
  \providecommand{\doi}{doi: \begingroup \urlstyle{rm}\Url}\fi

\bibitem[Ahmed and Atamt{\"u}rk(2011)]{AA:utility}
S.~Ahmed and A.~Atamt{\"u}rk.
\newblock Maximizing a class of submodular utility functions.
\newblock \emph{Mathematical Programming}, 128:\penalty0 149--169, 2011.

\bibitem[Alizadeh(1995)]{Alizadeh1995}
F.~Alizadeh.
\newblock Interior point methods in semidefinite programming with applications
  to combinatorial optimization.
\newblock \emph{SIAM Journal on Optimization}, 5:\penalty0 13--51, 1995.

\bibitem[Alizadeh and Goldfarb(2003)]{Alizadeh2003}
F.~Alizadeh and D.~Goldfarb.
\newblock Second-order cone programming.
\newblock \emph{Mathematical Programming}, 95:\penalty0 3--51, 2003.

\bibitem[Atamt\"urk and Gom\'ez(2016)]{AG:mixed-polymatroid}
A.~Atamt\"urk and A.~Gom\'ez.
\newblock Submodularity in conic quadratic mixed 0-1 optimization.
\newblock \emph{arXiv preprint arXiv:1705.05918}, 2016.
\newblock BCOL Research Report 16.02, UC Berkeley.

\bibitem[Atamt\"urk and Jeon(2017)]{AJ:lifted-polymatroid}
A.~Atamt\"urk and H.~Jeon.
\newblock Lifted polymatroid inequalities for mean-risk optimization with
  indicator variables.
\newblock \emph{arXiv preprint arXiv:1705.05915}, 2017.
\newblock BCOL Research Report 17.01, UC Berkeley.

\bibitem[Atamt{\"u}rk and Narayanan(2007)]{AN:conicmir}
A.~Atamt{\"u}rk and V.~Narayanan.
\newblock Cuts for conic mixed-integer programming.
\newblock In Matteo Fischetti and David~P. Williamson, editors, \emph{Integer
  Programming and Combinatorial Optimization}, pages 16--29, Berlin,
  Heidelberg, 2007. Springer.
\newblock ISBN 978-3-540-72792-7.

\bibitem[Atamt{\"u}rk and Narayanan(2008)]{Atamturk2008a}
A.~Atamt{\"u}rk and V.~Narayanan.
\newblock Polymatroids and risk minimization in discrete optimization.
\newblock \emph{Operations Research Letters}, 36:\penalty0 618--622, 2008.

\bibitem[Atamt{\"u}rk and Narayanan(2009)]{Atamturk2009}
A.~Atamt{\"u}rk and V.~Narayanan.
\newblock The submodular 0-1 knapsack polytope.
\newblock \emph{Discrete Optimization}, 6:\penalty0 333--344, 2009.

\bibitem[Atamt{\"u}rk et~al.(2017)Atamt{\"u}rk, Deck, and
  Jeon]{ADJ:mr-interdiction}
Alper Atamt{\"u}rk, Carlos Deck, and Hyemin Jeon.
\newblock Successive quadratic upper-bounding for discrete mean-risk
  minimization and network interdiction.
\newblock \emph{arXiv preprint arXiv:1708.02371}, 2017.
\newblock BCOL Reseach Report 17.05, UC Berkeley. Forthcoming in INFORMS
  Journal on Computing.

\bibitem[Aybat and Iyengar(2011)]{aybat2011first}
N.~S. Aybat and G.~Iyengar.
\newblock A first-order smoothed penalty method for compressed sensing.
\newblock \emph{SIAM Journal on Optimization}, 21:\penalty0 287--313, 2011.

\bibitem[Belloni et~al.(2011)Belloni, Chernozhukov, and
  Wang]{belloni2011square}
A.~Belloni, V.~Chernozhukov, and L.~Wang.
\newblock Square-root lasso: pivotal recovery of sparse signals via conic
  programming.
\newblock \emph{Biometrika}, 98:\penalty0 791--806, 2011.

\bibitem[Belotti et~al.(2013)Belotti, Kirches, Leyffer, Linderoth, Luedtke, and
  Mahajan]{jeff-minlp-review}
P.~Belotti, C.~Kirches, S.~Leyffer, J.~Linderoth, J.~Luedtke, and A.~Mahajan.
\newblock Mixed-integer nonlinear optimization.
\newblock \emph{Acta Numerica}, 22:\penalty0 1–131, 2013.

\bibitem[Ben-Tal and Nemirovski(1998)]{BenTal1998}
A.~Ben-Tal and A.~Nemirovski.
\newblock Robust convex optimization.
\newblock \emph{Mathematics of Operations Research}, 23:\penalty0 769--805,
  1998.

\bibitem[Ben-Tal and Nemirovski(1999)]{BenTal1999}
A.~Ben-Tal and A.~Nemirovski.
\newblock Robust solutions of uncertain linear programs.
\newblock \emph{Operations Research Letters}, 25:\penalty0 1--13, 1999.

\bibitem[Ben-Tal and Nemirovski(2001)]{BTN:ModernOptBook}
A.~Ben-Tal and A.~Nemirovski.
\newblock \emph{Lectures on Modern Convex Optimization: Analysis, Algorithms,
  and Engineering Applications}.
\newblock MPS-SIAM Series on Optimization. SIAM, Philadelphia, 2001.

\bibitem[Ben-Tal et~al.(2009)Ben-Tal, El~Ghaoui, and Nemirovski]{book:ro}
A.~Ben-Tal, L.~El~Ghaoui, and A.~Nemirovski.
\newblock \emph{Robust optimization}.
\newblock Princeton University Press, 2009.

\bibitem[Bertsimas and Sim(2004)]{Bertsimas2004}
D.~Bertsimas and M.~Sim.
\newblock Robust discrete optimization under ellipsoidal uncertainty sets,
  2004.

\bibitem[Bertsimas et~al.(2016)Bertsimas, King, Mazumder,
  et~al.]{bertsimas2016best}
D.~Bertsimas, A.~King, R.~Mazumder, et~al.
\newblock Best subset selection via a modern optimization lens.
\newblock \emph{The Annals of Statistics}, 44:\penalty0 813--852, 2016.

\bibitem[Bienstock(1996)]{bienstock1996computational}
D.~Bienstock.
\newblock Computational study of a family of mixed-integer quadratic
  programming problems.
\newblock \emph{Mathematical Programming}, 74:\penalty0 121--140, 1996.

\bibitem[Borchers and Mitchell(1994)]{Borchers1994}
B.~Borchers and J.~E. Mitchell.
\newblock An improved branch and bound algorithm for mixed integer nonlinear
  programs.
\newblock \emph{Computers \& Operations Research}, 21:\penalty0 359--367, 1994.

\bibitem[\c{C}ay et~al.(2017)\c{C}ay, P\'olik, and Terlaky]{CPT:warmstart}
S.~B. \c{C}ay, I.~P\'olik, and T.~Terlaky.
\newblock Warm-start of interior point methods for second order cone
  optimization via rounding over optimal {J}ordan frames, May 2017.
\newblock ISE Technical Report 17T-006, Lehigh University.

\bibitem[Dantzig et~al.(1955)Dantzig, Orden, and Wolfe]{Dantzig1955}
G.~B. Dantzig, A.~Orden, and P.~Wolfe.
\newblock The generalized simplex method for minimizing a linear form under
  linear inequality restraints.
\newblock \emph{Pacific Journal of Mathematics}, 5:\penalty0 183--196, 1955.

\bibitem[Dinh et~al.(2016)Dinh, Fukasawa, and Luedtke]{dinh2016exact}
T~Dinh, R~Fukasawa, and J~Luedtke.
\newblock Exact algorithms for the chance-constrained vehicle routing problem.
\newblock In \emph{International Conference on Integer Programming and
  Combinatorial Optimization}, pages 89--101. Springer, 2016.

\bibitem[Efron et~al.(2004)Efron, Hastie, Johnstone, Tibshirani,
  et~al.]{efron2004least}
B~Efron, T~Hastie, I~Johnstone, R~Tibshirani, et~al.
\newblock Least angle regression.
\newblock \emph{The Annals of Statistics}, 32:\penalty0 407--499, 2004.

\bibitem[{{E}l~{G}haoui} et~al.(2003){{E}l~{G}haoui}, Oks, and
  Oustry]{EOO:worst-var}
L.~{{E}l~{G}haoui}, M.~Oks, and F.~Oustry.
\newblock Worst-case value-at-risk and robust portfolio optimization: A conic
  programming approach.
\newblock \emph{Operations Research}, 51:\penalty0 543--556, 2003.

\bibitem[Hiriart-Urruty and Lemar{\'e}chal(2013)]{book:HUL-conv}
J.-B. Hiriart-Urruty and C.~Lemar{\'e}chal.
\newblock \emph{Convex Analysis and Minimization Algorithms I: Fundamentals},
  volume 305.
\newblock Springer Science \& Business Media, 2013.

\bibitem[Ishii et~al.(1981)Ishii, Shiode, Nishida, and Namasuya]{Ishii1981}
H.~Ishii, S.~Shiode, T.~Nishida, and Y.~Namasuya.
\newblock Stochastic spanning tree problem.
\newblock \emph{Discrete Applied Mathematics}, 3:\penalty0 263--273, 1981.

\bibitem[Karmarkar(1984)]{Karmarkar1984}
N.~Karmarkar.
\newblock A new polynomial-time algorithm for linear programming.
\newblock In \emph{Proceedings of the Sixteenth Annual ACM Symposium on Theory
  of Computing}, pages 302--311. ACM, 1984.

\bibitem[Leyffer(2001)]{Leyffer2001}
S.~Leyffer.
\newblock Integrating {SQP} and branch-and-bound for mixed integer nonlinear
  programming.
\newblock \emph{Computational Optimization \& Applications}, 18:\penalty0
  295--309, 2001.

\bibitem[Lobo et~al.(1998)Lobo, Vandenberghe, Boyd, and Lebret]{Lobo1998}
M.~S. Lobo, L.~Vandenberghe, S.~Boyd, and H.~Lebret.
\newblock Applications of second-order cone programming.
\newblock \emph{Linear Algebra \& its Applications}, 284:\penalty0 193--228,
  1998.

\bibitem[Megiddo(1991)]{Megiddo1991}
N.~Megiddo.
\newblock On finding primal- and dual-optimal bases.
\newblock \emph{INFORMS Journal on Computing}, 3:\penalty0 63--65, 1991.

\bibitem[Nemirovski and Scheinberg(1996)]{Nemirovskii1996}
A.~Nemirovski and K.~Scheinberg.
\newblock Extension of {K}armarkar's algorithm onto convex quadratically
  constrained quadratic problems.
\newblock \emph{Mathematical Programming}, 72:\penalty0 273--289, 1996.

\bibitem[Nesterov(2005)]{nesterov2005smooth}
Y~Nesterov.
\newblock Smooth minimization of non-smooth functions.
\newblock \emph{Mathematical Programming}, 103:\penalty0 127--152, 2005.

\bibitem[Nesterov and Nemirovski(1994)]{Nesterov1994}
Y.~Nesterov and A.~Nemirovski.
\newblock \emph{Interior-Point Polynomial Algorithms in Convex Programming}.
\newblock Society for Industrial and Applied Mathematics, 1994.
\newblock \doi{10.1137/1.9781611970791}.
\newblock URL \url{http://epubs.siam.org/doi/abs/10.1137/1.9781611970791}.

\bibitem[Nesterov and Todd(1998)]{Nesterov1998}
Y.~E. Nesterov and M.~J. Todd.
\newblock Primal-dual interior-point methods for self-scaled cones.
\newblock \emph{SIAM Journal on Optimization}, 8:\penalty0 324--364, 1998.

\bibitem[Nikolova et~al.(2006)Nikolova, Kelner, Brand, and
  Mitzenmacher]{nikolova2006stochastic}
E.~Nikolova, J~A Kelner, M~Brand, and M~Mitzenmacher.
\newblock Stochastic shortest paths via quasi-convex maximization.
\newblock In \emph{European Symposium on Algorithms}, pages 552--563. Springer,
  2006.

\bibitem[Tawarmalani and Sahinidis(2005)]{Tawarmalani2005}
M.~Tawarmalani and N.~V. Sahinidis.
\newblock A polyhedral branch-and-cut approach to global optimization.
\newblock \emph{Mathematical Programming}, 103\penalty0 (2):\penalty0 225--249,
  2005.

\bibitem[Tibshirani(1996)]{tibshirani1996regression}
R.~Tibshirani.
\newblock Regression shrinkage and selection via the lasso.
\newblock \emph{Journal of the Royal Statistical Society. Series B
  (Methodological)}, pages 267--288, 1996.

\bibitem[Van~de Panne and Whinston(1964)]{VanDePanne1964}
C.~Van~de Panne and A.~Whinston.
\newblock Simplicial methods for quadratic programming.
\newblock \emph{Naval Research Logistics Quarterly}, 11\penalty0
  (3-4):\penalty0 273--302, 1964.

\bibitem[Vielma et~al.(2015)Vielma, Dunning, Huchette, and Lubin]{Vielma2015}
J.~P. Vielma, I.~Dunning, J.~Huchette, and M.~Lubin.
\newblock Extended formulations in mixed integer conic quadratic programming.
\newblock \emph{Mathematical Programming Computation}, 2015.
\newblock Forthcoming. doi:10.1007/s12532-016-0113-y.

\bibitem[Wolfe(1959)]{Wolfe1959}
P.~Wolfe.
\newblock The simplex method for quadratic programming.
\newblock \emph{Econometrica: Journal of the Econometric Society}, pages
  382--398, 1959.

\bibitem[Yildirim and Wright(2002)]{YW:warmstart}
E.~A. Yildirim and S.~J. Wright.
\newblock Warm-start strategies in interior-point methods for linear
  programming.
\newblock \emph{SIAM Journal on Optimization}, 12:\penalty0 782--810, 2002.

\end{thebibliography}

%\pagebreak

\appendix

\section{Branch-and-bound algorithm}
\label{sec:branchAndBound}

Algorithm~\ref{alg:branchAndBound} describes the branch-and-bound algorithm used in computations. Throughout the algorithm, we maintain a list $L$ of the nodes to be processed. Each node is a tuple $(S,B,lb)$, where $S$ is the subproblem, $B$ is a basis for warm starting the continuous solver and $lb$ is a lower bound on the objective value of $S$. In line~\ref{line:initL} list $L$ is initialized with the root node. For each node, the algorithm calls a continuous solver (line~\ref{line:solverOracle}) which returns a tuple $(x,\bar{B},z)$, where $x$ is an optimal solution of $S$, $\bar{B}$ is the corresponding optimal basis and $z$ is the optimal objective value (or $\infty$ if $S$ is infeasible). The algorithm then checks whether the node can be pruned (lines \ref{line:bound1}-\ref{line:bound2}), $x$ is integer (lines \ref{line:integer1}-\ref{line:integer2}), or it further branching is needed (lines \ref{line:branch1}-\ref{line:branch2}).

\begin{algorithm}[h]
\caption{Branch-and-bound algorithm}
\label{alg:branchAndBound}
\begin{algorithmic}[1]
\renewcommand{\algorithmicrequire}{\textbf{Input:}}
\renewcommand{\algorithmicensure}{\textbf{Output:}}
\Require $P \text{, discrete minimization problem}$
\Ensure Optimal solution $x^*$
\State $ub\gets \infty$ \Comment{Upper bound}
\State $x^*\gets \emptyset$ \Comment{Best solution found}
\State $L\gets \left\{(P,\emptyset,-\infty)\right\}$ \label{line:initL} \Comment{list of nodes $L$ initialized with the original problem}

\While{$L\neq \emptyset$} \label{line:repeat}
\State $(S,B,lb)\gets \texttt{PULL}(L)$ \Comment{select and remove one element from $L$} \label{line:pull}

\If{$lb \geq ub$}\State \textbf{go to} line~\ref{line:repeat} \EndIf

\State $(x,\bar{B},z)\gets \texttt{SOLVE}(S,B)$ \Comment{solve continuous relaxation}\label{line:solverOracle}

\If{$z \geq ub$} \Comment{if $S$ is infeasible then $z=\infty$} \label{line:bound1}
\State \textbf{go to} line~\ref{line:repeat} \Comment{prune by infeasibility or bounds} \label{line:bound2}

\ElsIf{$x$ is integer} \label{line:integer1}
\State $ub\gets z$ \Comment{update incumbent solution}
\State $x^*\gets x$
\State \textbf{go to} line~\ref{line:repeat} \label{line:integer2} \Comment{prune by integer feasibility}

\Else\label{line:branch1}
\State $(S_{\leq},S_{\geq})\gets \texttt{BRANCH}(x)$  \Comment{create two subproblems}\label{line:branching}
\State $L\gets L\cup \left\{(S_{\leq},\bar{B},z) ,(S_{\geq},\bar{B},z) \right\}$ \Comment{add the subproblems to $L$}\label{line:branch2}
 \EndIf

\EndWhile
\State \Return $x^*$
\end{algorithmic}
\end{algorithm}

We now describe the specific implementations of the different subroutines. For branching (line~\ref{line:branching}) we use the maximum infeasibility rule, which chooses the variable $x_i$ with value $v_i$ furtherest from an integer (ties broken arbitrarily). The subproblems $S_{\leq}$ and $S_{\geq}$ in line~\ref{line:branch2} are created by imposing the constraints $x_i\leq \lfloor v_i \rfloor$ and $x_i\geq \lceil v_i \rceil$, respectively. The \texttt{PULL} routine in line~\ref{line:pull} chooses, when possible, the child of the previous node which violates the bound constraint by the least amount, and chooses the node with the smallest lower bound when the previous node has no child nodes. The list $L$ is thus implemented as a sorted list ordered by the bounds, so that the \texttt{PULL} operation is done in $O(1)$ and the insertion is done in $O(\log |L|)$ (note that in line~\ref{line:branch2} we only add to the list the node that is not to be processed immediately). A solution $x$ is assumed to be integer (line~\ref{line:integer1}) when the values of all variables are within $10^{-5}$ of an integer. Finally, the algorithm is terminated when  $\frac{ub-lb_{best}}{\left|lb_{best}+10^{-10}\right|}\leq 10^{-4}$, where $lb_{best}$ is the minimum lower bound among all the nodes in the tree. 

The maximum infeasibility rule is chosen due to its simplicity. The other rules and parameters correspond to the ones used in CPLEX branch-and-bound algorithm in default configuration.

\end{document}